\documentclass[11pt, a4paper, reqno, english]{amsart}

\usepackage{amsfonts, amsthm, amsmath, amscd, amssymb, mathrsfs}
\usepackage{enumitem}
\usepackage{comment}
\usepackage[all]{xy} 
\usepackage[draft=false]{hyperref}

\newtheorem{teo}{Theorem}
\newtheorem{prop}[teo]{Proposition}
\newtheorem{lem}[teo]{Lemma}

\theoremstyle{remark}
\newtheorem{remark}[teo]{Remark}

\newcommand{\C}{\mathbb{C}} 
\newcommand{\R}{\mathbb{R}} 
\newcommand{\Q}{\mathbb{Q}} 
\newcommand{\Z}{\mathbb{Z}} 
\newcommand{\G}{\mathbb{G}} 
\newcommand{\A}{\mathbb{A}} 

\newcommand{\K}{k} 
\newcommand{\Ok}{\mathcal{O}_\K} 
\newcommand{\idk}{\mathcal{I}} 
\newcommand{\ncl}{h_{\K}} 
\newcommand{\plk}{{\Omega_k}} 
\newcommand{\plf}{{\Omega_f}} 
\newcommand{\pli}{{\Omega_{\infty}}} 
\newcommand{\norm}{N} 
\newcommand{\F}{\mathbb{F}} 
\newcommand{\rOk}{\Ok^{\times}} 

\newcommand{\aid}{\mathfrak{a}} 
\newcommand{\p}{\mathfrak{p}} 
\newcommand{\ua}{\underline{\mathfrak{c}}} 
\newcommand{\clrep}{\mathcal{C}} 
\newcommand{\ida}{\mathfrak{c}} 
\newcommand{\ud}{\underline{\mathfrak{d}}} 
\newcommand{\uo}{\underline{\mathfrak{o}}} 
\newcommand{\oid}{\mathfrak{o}} 
\newcommand{\did}{\mathfrak{d}} 
\newcommand{\bid}{\mathfrak{b}} 
\newcommand{\ub}{\underline{\mathfrak{b}}} 
\newcommand{\idcoord}{\ua_{x_\rho}} 
\newcommand{\uidcoord}{\ua_{\ux}} 

\newcommand{\ux}{\underline{x}} 
\newcommand{\uyi}{\underline y} 
\newcommand{\yi}{y} 
\newcommand{\setsum}{S(\B)} 

\newcommand{\X}{X} 
\newcommand{\torus}{T} 
\newcommand{\latt}{T_{\ast}} 
\newcommand{\lattR}{T_{\ast\R}}
\newcommand{\fan}{\Sigma} 
\newcommand{\ray}{\Sigma(1)} 
\newcommand{\rminus}{S} 
\newcommand{\maxc}{\Sigma_{max}} 
\newcommand{\torsor}{\mathcal{T}} 
\newcommand{\tortor}{\Lambda} 
\newcommand{\lattor}{\Lambda_{\ast}} 
\newcommand{\drho}{{D_{\rho}}} 
\newcommand{\dtrho}{{D_{\tilde\rho}}}
\newcommand{\usigma}{{D_\sigma}} 
\newcommand{\ns}{T_{NS}} 
\newcommand{\deq}{u} 
\newcommand{\can}{{-K}} 
\newcommand{\asrho}{{\alpha_{\sigma,\rho}}} 
\newcommand{\astrho}{{\alpha_{\sigma,\tilde\rho}}} 

\newcommand{\wtX}{\widetilde{X}}
\newcommand{\wttorsor}{\widetilde\torsor}

\newcommand{\awttorsor}{{}_{\ua}\widetilde\torsor} 
\newcommand{\awpi}{{}_{\ua}\tilde\pi} 

\newcommand{\hX}{H} 
\newcommand{\ma}{m} 
\newcommand{\na}{n} 
\newcommand{\hsup}{h} 

\newcommand{\mob}{\mu}
\newcommand{\partsum}{\sum_{\ud,\p}}
\newcommand{\exponent}{e}
\newcommand{\bb}{b} 
\newcommand{\bbb}{b'} 

\newcommand{\card}{N_{\torus,\hX,\K}} 
\newcommand{\asymptconst}{C_{\X,\hX,\K}} 
\newcommand{\f}{f} 
\newcommand{\ka}{\kappa} 
\newcommand{\B}{B} 
\newcommand{\crOk}{w_k} 
\newcommand{\ii}{\infty} 
\newcommand{\truedisck}{\Delta_k} 
\newcommand{\disck}{|\truedisck|} 
\newcommand{\pey}{\alpha} 

\newcommand{\firstmeasure}{\overline\Theta^1}
\newcommand{\secondmeasure}{\Theta_\ii}
\newcommand{\adeles}{\mathbf{A}_\K}
\newcommand{\borel}{m_\ii} 
\newcommand{\nuintegers}{\mathcal{O}_\nu} 

\newcommand{\Aa}{A_{\ua,\uo}}
\newcommand{\Aad}{A_{\ua,\ud}}
\newcommand{\Ca}{C_{\ua}}
\newcommand{\nAad}{\#\Aad}
\newcommand{\nCa}{\#\Ca}
\newcommand{\darho}{\delta_{\ua,\ud,\tilde\rho}}
\newcommand{\darhos}{\delta_{\sigma}}

\newcommand{\e}{e} 
\newcommand{\ee}{\epsilon} 
\newcommand{\constlemma}{D}
\newcommand{\Sum}{S} 

\newcommand{\constantAa}{C_2} 
\newcommand{\constt}{C} 
\newcommand{\boundnumbersigma}{A}
\newcommand{\constant}{C} 
\newcommand{\constantO}{C_0} 
\newcommand{\constRi}{C} 
\newcommand{\constantpartdarho}{C_1} 
\newcommand{\constB}{B} 

\newcommand{\Da}{{D_{\ua}}} 
\newcommand{\Das}{{D_{\ua,\sigma}}} 
\newcommand{\Dasp}{{D_{\ua,\sigma'}}} 
\newcommand{\DasP}{{D'_{\ua,\sigma}}} 

\newcommand{\daux}{\,\mathrm{d}_{\ua,\uo}\ux} 
\newcommand{\dadux}{\,\mathrm{d}_{\ua,\ud}\ux} 
\newcommand{\partitionfunction}{\mathcal{L}} 
\newcommand{\Csigma}{C_{\sigma}} 
\newcommand\dd{\,\mathrm{d}} 

\newcommand{\lattice}{L} 
\newcommand{\vll}{v} 
\newcommand{\wll}{w} 
\newcommand{\ull}{u} 
\newcommand{\rll}{\eta} 
\newcommand{\thll}{\theta} 
\newcommand{\Fund}{F} 
\newcommand{\Funi}{F_\ui} 
\newcommand{\gi}{\gamma_\ui} 
\newcommand{\g}{\gamma} 

\newcommand{\NN}{I} 
\newcommand{\ui}{{\underline{i}}} 
\newcommand{\Li}{L_\ui} 
\newcommand{\Aadi}{A_{\ui}}
\newcommand{\NaBid}{\#\Aadi} 
\newcommand{\SBi}{D_{\ui}} 
\newcommand{\boundarySBi}{\tilde{D}_{\ui}} 
\newcommand{\Ris}{L_{\ui,\rho}} 
\newcommand{\Ri}{\Ris(\B)} 
\newcommand{\Iid}{I_{\ui}} 
\newcommand{\Fis}{S_{\ui}} 
\newcommand{\FiB}{\Fis(\B)} 

\DeclareMathOperator{\pic}{Pic}
\DeclareMathOperator{\spec}{Spec}
\DeclareMathOperator{\cox}{Cox}

\DeclareMathOperator{\Hom}{Hom}

\begin{document}

\title[Imaginary quadratic points on toric varieties]{Imaginary quadratic points on toric varieties via universal torsors}

\author{Marta Pieropan} 

\address{Institut f\"ur Algebra, Zahlentheorie und Diskrete
  Mathematik, Leibniz Universit\"at Hannover,
  Welfengarten 1, 30167 Hannover, Germany}

\email{pieropan@math.uni-hannover.de}

\date{\today}

\setcounter{tocdepth}{1}

\begin{abstract}
Inspired by a paper of Salberger we give a new proof of Manin's conjecture for toric varieties over imaginary quadratic number fields by means of  universal torsor parameterizations and elementary lattice point counting. 
\end{abstract}


\maketitle

\tableofcontents

\section{Introduction}
Let $\K$ be a number field. 
Around 1989, Manin predicted an asymptotic formula for the distribution of rational points on Fano varieties over $k$: there exists an open subset $U$ of the variety $X$ such that
\begin{equation}\label{Manin'sconjecture}
\#\{x\in U:H(x)\leq\B\}\sim C_{X,H,\K}\B(\log\B)^{r-1},
\end{equation}
where $H:X(\K)\to\R_{\geq0}$  is the height function defined by an adelic metric on the anticanonical sheaf, $r$ is the rank of the Picard group of $X$, and $C_{X,H,\K}$ is a constant independent of $\B$ \cite{MR974910}. 
The constant $C_{X,H,\K}$ depends on $X$, $H$ and $\K$, and has been interpreted in terms of Tamagawa numbers by Peyre \cite{MR1340296, MR2019019}.

The asymptotic formula \eqref{Manin'sconjecture} above does not hold for all Fano varieties (cf.~\cite{MR1401626}), but has been verified for some families of varieties and many specific examples. 
The proofs of the known cases make use of various techniques: the Hardy-Littlewood circle method applies to complete intersections whose dimension is large compared to the degree (e.g.~\cite{MR0150129, MR1340296}), harmonic analysis on adelic points works for certain equivariant compactifications of algebraic groups (e.g.~\cite{MR1620682, MR1906155}), other proofs employ torsor parameterizations and lattice point counting. 
This last idea goes under the name of ``the universal torsor method''. It was introduced  by Salberger, who proved Manin's conjecture for 
complete smooth split toric varieties over $\Q$ with anticanonical sheaf generated by global sections \cite{MR1679841}, and was then applied to many other varieties over $\Q$: for example, some (singular) del Pezzo surfaces  (e.g.~\cite{MR1909606, MR2373960, MR2332351, MR3100953}) and Ch\^atelet surfaces (e.g.~\cite{MR2874644}). 

The circle method and harmonic analysis work for families of varieties over arbitrary number fields, but require that the varieties satisfy very special properties. 
On the other hand, the universal torsor method applies to a wider range of varieties, but has been handled mainly over $\Q$. 
Only recently, a generalization of this method to other number fields was started 
\cite{MR3077170,  MR3107569, DF1MR3269462, DF2MR3181632, DF3arXiv:1311.2809, arXiv:1312.6603}.

\bigskip

Universal torsors  were introduced and studied by Colliot-Th\'el\`ene and Sansuc \cite{MR0414556, MR899402} to investigate the rational points on geometrically rational vartieties.
A universal torsor of $X$ is a geometric quotient $\pi:\torsor\to X$ of $\torsor$ by 
the N\'eron-Severi torus $\ns$ of $X$, which is the group of $\K$-characters of the geometric Picard group $\pic(\overline X)$. For a complete variety with $\pic(\overline X)$ free and split (i.e.~with trivial Galois action) of rank $r$ and with
a proper model $\wtX$ over the ring of integers $\Ok$ of $\K$, the set of rational points $\X(\K)$ coincides with the set of integral points $\wtX(\Ok)$ and is parameterized by the integral points on the universal torsors $\tilde\pi:\wttorsor\to\wtX$ of $\wtX$. The fiber of $\tilde\pi$ over each $\Ok$-point is either empty or isomorphic to $(\Ok^\times)^r$, hence counting rational points on $X$ is the same as counting integral points on the torsors $\wttorsor$ 
modulo the action of $(\Ok^\times)^r$. 

There are at least two reasons why so far the universal torsor method has been applied mainly over $\Q$. Firstly, because over $\Z$ there is just one isomorphism class of torsors $\wttorsor$ in the split case, while over number fields of class number greater than 1 one has usually to count lattice points on more than one torsor $\wttorsor$. Secondly, because the fibers of a universal torsor morphism over $\Z$ are finite of the same cardinality $\#(\Z^\times)^r=2^r$. Hence, the number of points on a subset of $\wtX$ is obtained by counting the number of points on its preimage under $\tilde\pi$ and then dividing by $2^r$. A similar trick works if $\K$ is an imaginary quadratic field. This explains why the first applications of the universal torsor method to number fields beyond $\Q$ \cite{MR3077170, DF1MR3269462, DF2MR3181632, DF3arXiv:1311.2809} consider imaginary quadratic fields. 

If $\Ok^\times$ is not finite, one needs to work with a fundamental domain for the action of $(\Ok^\times)^r$ on $\wttorsor(\Ok)$ (see \cite{MR3107569}). This leads sometimes to count lattice points in unbounded regions that require refined techniques \cite{arXiv:1312.6603}.
This last paper
offers a systematic and explicit way to produce parameterizations in terms of lattice points on twisted universal torsors for varieties over number fields with arbitrary class number and applies it to a specific singular del Pezzo surface of degree 4. 

\bigskip

Toric varieties are a classical topic in the literature about Manin's conjecture. 
We recall that Schanuel's asymptotic formula for the number of points of bounded height on projective spaces of any dimension over arbitrary number fields \cite{MR557080} inspired Manin's conjecture \cite{MR974910}.
Smooth projective toric varieties have been worked out via harmonic analysis over arbitrary number fields \cite{MR1620682} and over global fields of positive characteristic \cite{MR2809202}. Salberger considered 
complete smooth split toric varieties over $\Q$ with anticanonical sheaf generated by global sections 
to test the universal torsor parameterization and other innovations in his remarkable paper \cite{MR1679841}, the error term in his result was then improved by de la Bret\`eche in \cite{MR1824152} by means of the precise estimations of multiple sums of arithmetic functions he developed in \cite{MR1858338}.

In our paper we generalize Salberger's proof to imaginary quadratic fields. 
It is the first proof of Manin's conjecture via the universal torsor method for a family of varieties over number fields beyond $\Q$. Moreover, like Salberger's and de la Bret\`eche's results, our main theorem below covers some complete non-projective toric varieties (for example Oda's threefold \cite[\S1.9.4]{MR546291}) that have been neglected by \cite{MR1620682} and \cite{MR2809202}.

\bigskip

The main result of this paper is the following theorem.

\begin{teo}\label{maintheo}
Let $\X$ be a smooth equivariant compactification of a split torus $\torus$ over an imaginary quadratic number field $\K$.
Assume that the anticanonical sheaf of $\X$ is generated by its global sections. Let $\hX$ be the toric anticanonical height function on $\X(k)$ defined in \cite[\S10]{MR1679841}, and let $\card(\B)$ be the number of $\K$-rational points on $\torus$ of toric anticanonical height at most $\B$. Then, for all $\varepsilon>0$,
\begin{equation*} \label{result}
\card(\B)=\asymptconst\B(\log\B)^{r-1} +O(\B(1+\log\B)^{r-2+1/f+\varepsilon}),
\end{equation*}
where $r$ is the rank of the Picard group of $\X$, $\f$ is the smallest positive integer such that there exist $\f$ rays of the fan $\fan$ defining $\X$ not contained in a cone of $\fan$, and $\asymptconst$ is the constant predicted by Peyre in \cite{MR1340296}.
\end{teo}

The constant $\asymptconst$ is the product
\begin{equation}\label{eq:toric_asymptconst}
\pey(\X)\ka\disck^{-N/2}\ncl^r\crOk^{-r}(2\pi)^N\#\maxc,
\end{equation}
where 
$\pey(\X)$ is the volume of a certain polytope in the dual of the effective cone of $\X$ defined in \cite[\S2]{MR1340296} (cf.~\cite[\S7]{MR1679841}),
$\ka$ is a product of non-archimedean local densities (cf.~\eqref{kappa} and Proposition \ref{prop:peyre_constant}),
$\truedisck$ is the discriminant of $\K$,
$\ncl$ is the class number of $\K$,  
$\crOk$ is the number of roots of 1 contained in $\K$, and $N$ and
$\#\maxc$ are the number of rays and  the number of maximal cones, respectively, in the fan defining $\X$.

This paper is organized as follows.
In Section \ref{sec:toric_varieties} 
we recall the Cox rings of split toric varieties, which are polynomial rings by \cite{MR1299003}, and we describe the the associated universal torsors, which are  open subsets of affine spaces whose complements are defined by monomial equations.
These equations are determined by the combinatorics of the toric varieties, and turn into coprimality conditions on the affine coordinates of the integral points on the twisted universal torsors that we use to parameterize rational points via lattice points. 
The anticanonical height function $\hX$ defined in \cite{MR1679841} is the
pullback of the Weil height
\begin{equation*}
H_{\mathbb{P}^n}:\mathbb{P}^n(\K)\to\R_{\geq0},\quad (x_0:\dots:x_n)\mapsto\prod_{\nu\in\plk}\max_{0\leq i\leq n}|x_i|_\nu,
\end{equation*}
where $\plk$ is the set of places of $\K$ and $|\cdot|_\nu$ are the absolute values on the completions $\K_\nu$ of $\K$ normalized as in \S \ref{notation}, under a morphism $\X\to\mathbb{P}^n_k$ defined by the anticanonical sheaf.
We prove that the height of the integral points on the twisted universal torsors can be expressed in terms of the coordinates of the affine spaces containing the torsors as a product involving just the archimedean places (cf.~Proposition \ref{height2}).

In Section \ref{sec:moebius_inversion} we study a multiplicative function, introduced in \cite{MR1340296} and used also in \cite{MR1679841}, attached to the characteristic function of the set of integral coordinates that belong to the universal torsor, and we  perform M\"obius inversion to get rid of the coprimality conditions. Thus, we reduce to count lattice points in bounded subsets $D(\B)$ of an $\R$-vector space $\C^N\cong\R^{2N}$.

Section \ref{sec:proof_main_theorem} contains the proof of Theorem \ref{maintheo}. 
Since the height function depends only on the absolute value of the coordinates, we partition $\C^N$ in strongly convex rational polyhedral cones, each of them spanned by a fundamental domain of the lattice which is a parallelotope $F$ 
with the property that, for every lattice point $x$ in the cone, $x+F$ intersects the boundary of the region $D(\B)$ defined by the height function if and only if the lifted height of $x$ is at most $\B$ and the height of $x+\gamma$ is strictly greater than $\B$, where $\gamma$ is the long diagonal of the parallelotope (cf.~Subsection \ref{subsec:lattices}).  
For each cone $C$, we compare the number of lattice points in $C\cap D(\B)$ with the volume of the region $C\cap D(\B)$ (cf.~Proposition \ref{lemnAaBint}), and we estimate the difference by counting the number of lattice points $x\in C\cap D(\B)$ such that $x+F\nsubseteq C\cap D(\B)$ (cf.~Proposition \ref{partdarho}).

In the last section we show that the constant $\asymptconst$ obtained in Section \ref{sec:proof_main_theorem} and described above verifies Peyre's conjecture \cite[Conjecture 2.3.1]{MR1340296}. 

The results of Sections \ref{sec:toric_varieties} and \ref{sec:moebius_inversion} up to \S \ref{subsec:moebius_inversion_imaginary_quadratic} hold for arbitrary number fields. Then we restrict to imaginary quadratic fields, because a generalization of the proofs in \cite{MR557080, MR3107569} does not seem to be straightforward.

\subsection{Notation}\label{notation}
Let $\K$ be a number field. Let $\plk$ be the set of places of $\K$, $\plf$ the set of finite places of $\K$ and $\pli$ the set of infinite places of $\K$. For every $\nu\in\plk$, we denote by $\K_{\nu}$ the completion of $\K$ at the place $\nu$. For every $\nu\in\plf$, if $\p$ is the prime ideal of $\Ok$ corresponding to $\nu$, we denote by $\F_{\p}\cong\F_{\nu}=\Ok/\p$ the residue field of $\K_{\nu}$.

For every $\nu\in\plk$, let $|\cdot|_{\nu}:\K_{\nu}\to\R_{\geq0}$ be the absolute value of $\K_{\nu}$ normalized as follows: if $\tilde\nu$ is the place of $\Q$ below $\nu$ and $\Q_{\tilde\nu}$ is the completion of $\Q$ at $\tilde\nu$, then $|\cdot|_{\nu}:=|\norm_{\K_\nu/\Q_{\tilde\nu}}(\cdot)|_{\tilde\nu}$, where $|\cdot|_{\tilde\nu}$ is the usual real or $p$-adic absolute value on $\Q_{\tilde\nu}$.

For every nonzero fractional ideal $\aid$ of $\K$, we denote by $\norm(\aid)\in\Q$ its absolute norm.

If $L_1,\dots,L_n$ are lattices (or fundamental domains of lattices) in $\C\cong\R^2$, we denote by $\bigoplus_{i=1}^nL_i$ the lattice (respectively, the fundamental domain) in $\C^n\cong\R^{2n}$ obtained as direct sum of the $L_i$.

\section{Universal torsors and heights on toric varieties}\label{sec:toric_varieties}

\subsection{Parameterization via universal torsors}

Let $\X$ be a smooth complete toric variety over a number field $\K$ that is an equivariant compactification of a split torus $\torus$. We denote by $\latt$ the lattice of cocharacters of $\torus$ and by $\fan$ the fan in $\lattR:=\latt\otimes_{\Z}\R$ that defines $\X$. 

By \cite[\S4]{MR0414556} and \cite[\S8]{MR1679841}, we know that $\X$ has a universal torsor $\pi:\torsor\to\X$, which is unique up to isomorphism and can be realized as the the toric variety defined by the pullback of $\fan$ under the morphism of lattices \begin{equation*}
\lattor\to\latt,\quad \rho\mapsto n_\rho,
\end{equation*}
where $n_\rho$ is the unique generator of $\rho\cap\latt$, $\lattor:=\bigoplus_{\rho\in\ray}\Z\rho$ and $\ray$ denotes the set of rays of $\fan$. We denote by  $\tortor=\pi^{-1}(\torus)$ the torus contained in $\torsor$ whose group of cocharacters is $\lattor$.
By \cite{MR1299003} the Cox ring of $\X$ is a $\pic(\X)$-graded polynomial ring in $N:=\#\ray$ variables:
\begin{equation*}\cox(X)=k[x_{\rho}:\rho\in\ray],
\end{equation*}
 where $\ray$ is the set of rays of $\fan$, and the degree of the variable $x_\rho$ is the class of the prime toric invariant divisor $\drho$ corresponding to the ray $\rho$.
 The Cox ring of $\X$ is the ring of global sections of the structure sheaf of $\torsor$ (see \cite[\S 6.1]{arXiv:1003.4229} and \cite[\S3.2]{arXiv:1408.5358}). Moreover,
 there is an open embedding  $\torsor\to\A^N_\K:=\spec\cox(\X)$  whose complement is the closed subset defined by the monomials \begin{equation*}
 \ux^{\usigma}:=\prod_{\rho\in\ray\smallsetminus\sigma(1)}x_{\rho}, \quad \sigma\in\maxc,
 \end{equation*} 
where $\maxc$ is the set of maximal cones of $\fan$, and $\sigma(1)$ is the set of rays of $\fan$ contained in the cone $\sigma$ (see \cite[\S8]{MR1679841}).

Let $\Ok$ be the ring of integers of $\K$.
By \cite[Remark 8.6]{MR1679841} (cf.~\cite[Theorem 3.3]{arXiv:1312.6603}), the universal torsor $\torsor\to\X$ admits an  $\Ok$-model $\wttorsor\to\wtX$, which is a universal torsor of the toric $\Ok$-scheme $\wtX$ defined by the fan $\fan$, and 
\begin{equation}\wttorsor\cong\A^N_{\Ok}\smallsetminus V(\ux^{\usigma}:\sigma\in\maxc). \label{eq:open_subset}
\end{equation}

 Fix a basis $\ell_1,\dots,\ell_r$ of $\pic(\X)$ (namely, an isomorphism $\pic(\X)\cong\Z^r$).
 Let $\clrep$ be a system of representatives for the class group of $k$.
By \cite[Theorem 2.7]{arXiv:1312.6603}, the universal torsor $\wttorsor\to\wtX$ induces the following parameterization of the set of rational points of $\X$:
\begin{equation}\label{eq:param}
\X(\K)=\bigsqcup_{\ua\in\clrep^r}\awpi(\awttorsor(\Ok)),
\end{equation}
where  $\awpi:\awttorsor\to\wtX$ is the twist of $\wttorsor\to\wtX$ defined in \cite[Definition 2.6]{arXiv:1312.6603}. 
The parameterization can be made explicit as follows.
 For every $\ua=(\ida_1,\dots,\ida_r)\in\clrep^r$ and every divisor $D$ of $\X$, let $\ua^D:=\prod_{i=1}^r\ida_i^{a_i}$, where $a_1,\dots,a_r\in\Z$ are determined by the equality $[D]=\sum_{i=1}^ra_i\ell_i$ in $\pic(\X)$. 
 For every $\ux=(x_\rho)_{\rho\in\ray}\in\K^N$ and $\ua\in\clrep^r$, let $\uidcoord=(\idcoord)_{\rho\in\ray}$ be defined by $\idcoord:=x_\rho\ua^{-\drho}$ for all $\rho\in\ray$.
 For every $\sigma\in\maxc$, let $\usigma:=\sum_{\rho\in\ray\smallsetminus\sigma(1)}\drho$ and $\uidcoord^{\usigma}:=\prod_{\rho\in\ray\smallsetminus\sigma(1)}\idcoord$.
With this notation,
\begin{equation}\label{eq:integral_points_twisted_torsor}
\awttorsor(\Ok)=\{\ux\in\bigoplus_{\rho\in\ray}\ua^\drho: \sum_{\sigma\in\maxc}\uidcoord^{\usigma}=\Ok\}\subseteq\K^N
\end{equation}
for all  $\ua\in\clrep^r$, by \cite[Theorem 2.7]{arXiv:1312.6603} (cf.~\cite[p.~15]{MR1650339}).

\subsection{Anticanonical height function}\label{subsection_height}
We recall the anticanonical height function defined by Salberger in \cite[\S10]{MR1679841} and list some of its properties. We need first some notation.

For every torus invariant divisor $D=\sum_{\rho\in\ray}a_\rho\drho$ of $\X$ and for every $\sigma\in\maxc$, let $\deq_{\sigma,D}$ be the character of $\torus$ determined by $\deq_{\sigma,D}(n_\rho)=a_\rho$ for all $\rho\in\sigma(1)$, and define $D(\sigma):=D-\sum_{\rho\in\ray}\deq_{\sigma,D}(n_\rho)\drho$.

We denote by $\can$ the anticanonical divisor $\sum_{\rho\in\ray}\drho$ of $\X$. 
For every $\sigma\in\maxc$ and  $\rho\in\ray$, let $\asrho:=1-\deq_{\sigma,\can}(n_\rho)$, so that $\can(\sigma)=\sum_{\rho\in\ray}\asrho\drho$.
By \cite[Proposition 8.7 (a)]{MR1679841}, we know that if the anticanonical sheaf of $\X$ is generated by its global sections, then $\can(\sigma)$ is an effective divisor for all $\sigma\in\maxc$, and $\asrho=0$ for  all $\rho\in\sigma(1)$. From now on, we assume that the anticanonical sheaf of $\X$ is generated by its global sections.

Let $\nu\in\plk$.
 For every $\ux=(x_\rho)_{\rho\in\ray}\in\K_\nu^N$ and every  effective divisor $D=\sum_{\rho\in\ray}a_\rho\drho$ of $X$, we denote by $\ux^D$ the product $\prod_{\rho\in\ray}x_\rho^{a_\rho}$. 

Let $\hX:X(k)\to\R_{\geq0}$ be the height function corresponding to the anticanonical divisor $\can$ defined in \cite[(10.4)]{MR1679841}.
 For every $\ux\in\torsor(k)$,
\begin{equation}
\hX(\pi(\ux))=\prod_{\nu\in\plk}\sup_{\sigma\in\maxc}|\ux^{\can(\sigma)}|_{\nu}\label{eq:height_all_places}
\end{equation}
 by
 \cite[Proposition 10.14]{MR1679841}.
For integral points on the twisted torsors that appear in \eqref{eq:param}, the height \eqref{eq:height_all_places} can be expressed as product involving just the archimedean places as the following proposition shows.

\begin{prop}\label{height2}
Let $\ua\in\clrep^r$ and $\ux\in\awttorsor(\Ok)$, then \begin{equation*}
\hX(\pi(\ux))=\frac1{\norm(\ua^{\can})}\prod_{\nu\in\pli}\sup_{\sigma\in\maxc}|\ux^{\can(\sigma)}|_{\nu}.
\end{equation*}
\end{prop}
\begin{proof} 
For every $\rho\in\ray$, let $\ma_{\rho,\p},\na_{\rho,\p}\in\Z$ be defined  by the equalities $\ua^{\drho}=\prod_{\p}\p^{\ma_{\rho,\p}}$ and $x_\rho\Ok=\prod_{\p}\p^{\na_{\rho,\p}}$, where $\prod_\p$ denotes a product over the prime ideals of $\Ok$. Let $\nu\in\plf$ be a finite place associated to a prime ideal $\p$ of $\Ok$. 
 Then
\begin{equation*}\label{locp}
\sup_{\sigma\in\maxc}|\ux^{\can(\sigma)}|_{\nu}=\left(\frac1{\norm(\p)}\right)^{\min_{\sigma\in\maxc}\sum_{\rho\in\ray}\asrho\na_{\rho,\p}}.
\end{equation*}
We write
\begin{equation*}\begin{split}
\min_{\sigma\in\maxc}\sum_{\rho\in\ray}\asrho\na_{\rho,\p}&=\min_{\sigma\in\maxc}\left(\sum_{\rho\in\ray}\asrho(\na_{\rho,\p}-\ma_{\rho,\p})+\sum_{\rho\in\ray}\asrho\ma_{\rho,\p}\right)\\
&=\min_{\sigma\in\maxc}\sum_{\rho\in\ray\smallsetminus\sigma(1)}\asrho(\na_{\rho,\p}-\ma_{\rho,\p})+\sum_{\rho\in\ray}\ma_{\rho,\p},
\end{split}\end{equation*}
as  $\asrho=0$ for all $\rho\in\sigma(1)$ and $[\can(\sigma)]=[\can]$ in $\pic(\X)$ for all $\sigma\in\maxc$.

Since $x_{\rho}\in\ua^{\drho}$ for all $\rho\in\ray$, the inequality $\na_{\rho,\p}\geq\ma_{\rho,\p}$ holds for all $\rho\in\ray$.
The coprimality condition $\sum_{\sigma\in\maxc}\uidcoord^{\usigma}=\Ok$ gives
\begin{equation*}\label{eq:cocoprim}
\min_{\sigma\in\maxc}\sum_{\rho\in\ray\smallsetminus\sigma(1)}(\na_{\rho,\p}-\ma_{\rho,\p})=0.
\end{equation*}
Hence, $\min_{\sigma\in\maxc}\sum_{\rho\in\ray}\asrho\na_{\rho,\p}=\sum_{\rho\in\ray}\ma_{\rho,\p}$, as $\can(\sigma)$ is effective for all $\sigma\in\maxc$.
\end{proof}

\begin{lem} \label{11.20}
Let $\nu\in\plk$ and $\ux\in\torsor(k_{\nu})$ then
\begin{equation*}
|\ux^{\can}|_{\nu}\leq\sup_{\sigma\in\maxc}|\ux^{\can(\sigma)}|_{\nu}.
\end{equation*}
\end{lem}
\begin{proof}
By \cite[Proposition 9.2]{MR1679841} there exists a cone $\sigma\in\maxc$ such that $\left|\frac{\ux^{\can}}{\ux^{\can(\sigma)}}\right|_{\nu}\leq1$. Then $|\ux^{\can}|_{\nu}\leq|\ux^{\can(\sigma)}|_{\nu}\leq\sup_{\sigma\in\maxc}|\ux^{\can(\sigma)}|_{\nu}$.
\end{proof}

\section{M\"obius inversion}\label{sec:moebius_inversion}

\subsection{Generalized M\"obius function}\label{subsec:generalized_moebius_function}

We introduce a generalized M\"obius function, like in \cite[\S8.5]{MR1340296} and \cite[\S11]{MR1679841}, that we use to get rid of the coprimality condition  in \eqref{eq:integral_points_twisted_torsor} via M\"obius inversion. In order to do so, we fix some notation.

Let $\idk$ be the set of nonzero ideals of $\Ok$.
For every $\ud=(\did_{\rho})_{\rho\in\ray}\in\idk^N$, let $\norm(\ud):=\prod_{\rho\in\ray}\norm(\did_{\rho})\in\Z_{>0}$. 
For every prime ideal $\p$ of $\Ok$, we denote by $\partsum$ a sum over the set of $\ud=(\did_{\rho})_{\rho\in\ray}\in\idk^N$ such that $\did_{\rho}$ is a nonnegative power of $\p$ for all $\rho\in\ray$.
We denote by $\prod_{\p}$ a product over all prime ideals of $\Ok$, and by $\sum_{\norm(\ud)\leq\bb}$, $\sum_{\norm(\ud)\geq\bb}$ a sum over the set of $\ud\in\idk^N$ that satisfy $\norm(\ud)\leq\bb$, $\norm(\ud)\geq\bb$, respectively.

For every $\ub=(\bid_{\rho})_{\rho\in\ray}\in\idk^N$ and for every effective divisor $D$ of the form $\sum_{\rho\in\ray}a_\rho\drho$ on $\X$, let $\ub^D:=\prod_{\rho\in\ray}\bid_{\rho}^{a_\rho}$.
We denote by 
\begin{equation*}
\chi:\idk^N\to\{0,1\}
\end{equation*}
 the characteristic function of the subset
\begin{equation*}
\{\ub\in\idk^N:\sum_{\sigma\in\maxc}\ub^{\usigma}=\Ok\}\ \subseteq\ \idk^N.
\end{equation*}

For every $\ud\in\idk^N$, let 
\begin{equation*}
\chi_{\ud}:\idk^N\to\{0,1\}
\end{equation*} be the characteristic function of the subset
\begin{equation*}
\{\ub\in\idk^N: \ud|\ub\}\ \subseteq\ \idk^N,
\end{equation*}
where $\ud|\ub$ means $\bid_{\rho}\subseteq\did_{\rho}$ for all $\rho\in\ray$.

By \cite[ Lemme 8.5.1]{MR1340296}, there exists a unique function $\mob:\idk^N\to\Z$ such that \begin{equation*}\label{chimob}
\chi=\sum_{\ud\in\idk^N}\mob(\ud)\chi_{\ud},\end{equation*}
which satisfies, for every $\ud\in\idk^N$,
\begin{equation}\label{mobmult}
\mob(\ud)=\prod_{\p}\mob(\ud_{\p}),
\end{equation}  
where $\ud_\p:=(\did_{\rho,\p})_{\rho\in\ray}$ with $\did_{\rho,\p}$ a nonnegative power of $\p$ such that $\did_\rho=\prod_\p\did_{\rho,\p}$ holds in $\Ok$ for all $\rho\in\ray$.

The next proposition, analogous to \cite[Lemmas 11.15 and 11.19]{MR1679841}, shows some properties of $\mob$ that we use later.

\begin{prop}\label{mobius_function}  
Let  $\f\in\Z_{>0}$ be the smallest number of rays of $\fan$ not contained in a cone of $\fan$.
Then
\begin{enumerate}[label=(\roman{*}), ref=(\roman{*})]
\item the product $\prod_{\p}\partsum\frac{\mob(\ud)}{\norm(\ud)^s}$ and the sum $\sum_{\ud\in\idk^N}\frac{\mob(\ud)}{\norm(\ud)^s}$ converge absolutely and coincide for $s>1/\f$; \label{mobsumprod}
\item $\sum_{\norm(\ud)\leq\bb}|\mob(\ud)|=O(\bb^{1/f+\varepsilon})$, for all $\varepsilon>0$;\label{moblowpartsum}
\item $\sum_{\norm(\ud)\geq\bb}(|\mob(\ud)|/\norm(\ud))=O(\bb^{1/f-1+\varepsilon})$, for all $\varepsilon>0$.\label{mobuppartsum}
\end{enumerate}
\end{prop}
\begin{proof}
To prove \ref{mobsumprod}, we note that \eqref{mobmult} gives\begin{equation*}
\sum_{\substack{\ud\in\idk^N\\\norm(\ud)\leq\bb}}\frac{|\mob(\ud)|}{\norm(\ud)^s}=\sum_{\substack{\ud\in\idk^N\\\norm(\ud)\leq\bb}}\prod_{\p}\frac{|\mob(\ud_{\p})|}{\norm(\ud_{\p})^s}\leq\prod_{\p}\partsum\frac{|\mob(\ud)|}{\norm(\ud)^s}
\end{equation*} for all $s\in\R$ and  $\bb>0$, and 
\begin{equation*}
\lim_{\bb\to\infty}\sum_{\substack{\ud\in\idk^N\\\norm(\ud)\leq\bb}}\frac{|\mob(\ud)|}{\norm(\ud)^s}=\prod_{\p}\partsum\frac{|\mob(\ud)|}{\norm(\ud)^s}.
\end{equation*} 
Hence, it suffices to show that $\prod_{\p}\partsum\frac{|\mob(\ud)|}{\norm(\ud)^s}$ converges for $s>1/{\f}$.

Let $\p$ be a prime ideal of $\Ok$. The sum $\partsum\frac{|\mob(\ud)|}{\norm(\ud)}$ is finite,
because $\mob((\p^{\exponent_{\rho}})_{\rho\in\ray})=0$ for all $N$-tuples of non-negative integers
$(\exponent_{\rho})_{\rho\in\ray}$ such that $\exponent_{\rho}\geq2$ for some $\rho\in\ray$.
If $(\exponent_{\rho})_{\rho\in\ray}$ is an $N$-tuple of non-negative integers, not all 0, and  $\exponent_{\rho}=0$ for all $\rho\in\ray\smallsetminus\sigma(1)$ for a maximal cone $\sigma\in\maxc$, then $\mob((\p^{\exponent_{\rho}})_{\rho\in\ray})=0$. Therefore, if $(\exponent_{\rho})_{\rho\in\ray}$ is an $N$-tuple of non-negative integers, not all 0, such that $\mob((\p^{\exponent_{\rho}})_{\rho\in\ray})\neq0$, then $\norm(\p)^{\f}|\norm((\p^{\exponent_{\rho}})_{\rho\in\ray})$. Then we can write
\begin{equation*}
\partsum\frac{|\mob(\ud)|}{\norm(\ud)^s}=1+\frac1{\norm(\p)^{\f s}}Q\left(\frac1{\norm(\p)^s}\right),
\end{equation*}
where $Q$ is a polynomial with non-negative integer coefficients.
For every $N$-tuple $(\exponent_{\rho})_{\rho\in\ray}$ of non-negative integers, $\mob((\p^{\exponent_{\rho}})_{\rho\in\ray})$ does not depend on the choice of the prime ideal $\p$. Hence, $Q$
 is independent of the choice of $\p$. 
 Thus
\begin{equation*}
\sum_{\p}\frac1{\norm(\p)^{\f s}}Q\left(\frac1{\norm(\p)^s}\right)\leq [k:\Q]Q(1)\sum_{n\in\Z_{>0}}\frac1{n^{fs}}
\end{equation*}
is convergent for $s>1/{\f}$. Here $\sum_\p$ denotes a sum over the prime ideals of $\Ok$.

Property \ref{moblowpartsum} follows from \ref{mobsumprod} because for every $\varepsilon>0$, 
\begin{equation*}
\sum_{\norm(\ud)\leq\bb}|\mob(\ud)|\leq\bb^{1/f+\varepsilon}\sum_{\norm(\ud)\leq\bb}\frac{|\mob(\ud)|}{\norm(\ud)^{1/f+\varepsilon}}\leq\bb^{1/f+\varepsilon}\sum_{\ud\in\idk^N}\frac{|\mob(\ud)|}{\norm(\ud)^{1/f+\varepsilon}}.
\end{equation*}

Property \ref{mobuppartsum} can be proven using \ref{moblowpartsum} as follows:
\begin{equation*}\begin{split}
\sum_{\norm(\ud)\geq\bb}\frac{|\mob(\ud)|}{\norm(\ud)}&=\sum_{\bbb\geq\bb}\sum_{\norm(\ud)\leq\bbb}\frac{|\mob(\ud)|}{\bbb}-\sum_{\bbb\geq\bb-1}\sum_{\norm(\ud)\leq\bbb}\frac{|\mob(\ud)|}{\bbb+1}=\\
&=\sum_{\bbb\geq\bb}\sum_{\norm(\ud)\leq\bbb}\frac{|\mob(\ud)|}{\bbb(\bbb+1)}-\frac1{\bb}\sum_{\norm(\ud)\leq\bb-1}|\mob(\ud)|\leq\\
&\leq\sum_{\bbb\geq\bb}\sum_{\norm(\ud)\leq\bbb}\frac{|\mob(\ud)|}{\bbb^2}-\frac1{\bb}\sum_{\norm(\ud)\leq\bb-1}|\mob(\ud)|=\\
&=O\left(\sum_{\bbb\geq\bb}\bbb^{1/\f-2+\varepsilon}\right)+O\left(\frac{(\bb-1)^{1/\f+\varepsilon}}{\bb}\right). \qedhere
\end{split}\end{equation*}
\end{proof}

After Proposition \ref{mobius_function} \ref{mobsumprod} we define 
\begin{equation}\label{kappa}
\ka:=\prod_{\p}\partsum\frac{\mob(\ud)}{\norm(\ud)}=\sum_{\ud\in\idk^N}\frac{\mob(\ud)}{\norm(\ud)}.\end{equation}

\subsection{M\"obius inversion for imaginary quadratic fields}\label{subsec:moebius_inversion_imaginary_quadratic}

From now on we assume that $k$ is an imaginary quadratic extension of $\Q$. 
Let $\crOk$ be the cardinality of the group of units $\rOk$ of $\Ok$. 
We identify with $\C$ the completion of $\K$ at its only infinite place and we denote by $|\cdot|_{\ii}$ the corresponding absolute value normalized as in \S\ref{notation}.
%
For every $\ux\in\C^N$, let 
\begin{equation*}\hsup(\ux):=\sup_{\sigma\in\maxc}|\ux^{\can(\sigma)}|_\ii.\end{equation*}
For every $\ua\in\clrep^r$, every $\ud\in\idk^N$ and every $\B>0$, we define
\begin{align*}
\Aad(\B)&:=\{\ux\in\tortor(k)\cap\bigoplus_{\rho\in\ray}\did_\rho\ua^\drho:\hsup(\ux)\leq\norm(\ua^{\can})\B\},\\
\Ca(\B)&:=\{\ux\in\ua\wttorsor(\Ok)\cap\tortor(k):\hX(\pi(\ux))\leq\B\}.
\end{align*}
Let $\uo\in\idk^N$ be defined by $\oid_\rho:=\Ok$ for all $\rho\in\ray$. Then 
\begin{equation*}\Aad(\B)=\{\ux\in\Aa(\B): \ud|\uidcoord\}.\end{equation*} 

The sets defined above are finite, as the following proposition shows.

\begin{prop}\label{finiteness} 
The sets
$\Aad(\B)$ and $\Ca(\B)$ are finite for all $\ua\in\clrep^r$, 
all $\ud\in\idk^N$ and all $\B>0$. 
In particular $\Aad(\B)=\emptyset$ if $\norm(\ud)>\B$.
\end{prop}
\begin{proof} 
Let $\ua\in\clrep^r$, $\ud\in\idk^N$, $\B>0$ and
$\ux\in\Aa(\B)$. If $\ud|\uidcoord$, Lemma \ref{11.20} and the definition of $\can$ give 
\begin{equation*}
\norm(\ud)=\norm(\ud^{\can})\leq\norm(\uidcoord^{\can})=\frac{\norm(\ux^{\can})}{\norm(\ua^{\can})}\leq\frac{\hsup(\ux)}{\norm(\ua^{\can})}\leq\B,
\end{equation*}
and hence, $|x_\rho|_{\ii}\leq\norm(\did_\rho\ua^\drho)\norm(\ud)^{-1}\B$ for all $\rho\in\ray$, as $x_\rho\neq0$ for all $\rho\in\ray$.
Thus $\Aad(\B)$ is finite, and $\Aad(\B)=\emptyset$ if $\norm(\ud)>\B$. 
Moreover, $\Ca(\B)$ is finite as it is a subset of
$\Aa(\B)$ by \eqref{eq:integral_points_twisted_torsor} and Proposition \ref{height2}.
\end{proof}

Using the parameterizing property of universal torsors, we show that the counting function in Theorem \ref{maintheo} can be expressed in terms of the cardinalities of the sets $\Ca(\B)$ defined above.

\begin{prop}\label{cardexpr}
With $\card$ defined in Theorem \ref{maintheo} and the notation above, 
\begin{equation*}
\card(\B)=\frac1{\crOk^r}\sum_{\ua\in\clrep^r}\nCa(\B)
\end{equation*}
for all $\B>0$.
\end{prop}
\begin{proof}
Recall that $\card(\B)$ is the cardinality of the set
\begin{equation*}
\{\ux\in\torus(\K):\hX(\pi(\ux))\leq\B\}. 
\end{equation*}
Since $\pi^{-1}(\torus)=\tortor$, the parameterization \eqref{eq:param} gives
\begin{equation*}
\torus(\K)=\bigsqcup_{\ua\in\clrep^r}\awpi(\awttorsor(\Ok)\cap\tortor(\K)).
\end{equation*}
Let $\ua\in\clrep^r$. Since $\awpi$ is  a torsor over $\wtX$ under $\G_{m,\wtX}^r$ by \cite[Theorem 2.7]{arXiv:1312.6603}, its fibers over the $\Ok$-rational points of $\wtX$ are either empty or isomorphic to $(\rOk)^r$. Hence, each nonempty fiber of $\awpi$ over an $\Ok$-rational point of $\wtX$ contains exactly $\crOk^r$ points.
\end{proof}

In the next proposition we perform M\"obius inversion to reduce the proof of Theorem \ref{maintheo} to estimating the number of points in the sets $\Aad(\B)$.
\begin{prop}\label{CaB} For every $\ua\in\clrep^r$ and every $\B>0$,
\begin{equation*}
\nCa(\B)=\sum_{\ud\in\idk^N}\mob(\ud)\nAad(\B).
\end{equation*}
\end{prop}
\begin{proof}
The definition of $\chi$, $\mob$ and $\chi_{\ud}$, together with \eqref{eq:integral_points_twisted_torsor} and Proposition \ref{finiteness}, gives
\begin{align*}
\nCa(B)&
=\sum_{\ux\in\Aa(\B)}\chi(\uidcoord)
=\sum_{\ux\in\Aa(\B)}\sum_{\ud\in\idk^N}\mob(\ud)\chi_{\ud}(\uidcoord)=\\
&=\sum_{\ud\in\idk^N}\mob(\ud)\sum_{\ux\in\Aa(\B)}\chi_{\ud}(\uidcoord)
=\sum_{\ud\in\idk^N}\mob(\ud)\nAad(\B).\qedhere
\end{align*}
\end{proof}

\section{Proof of Theorem \ref{maintheo}}\label{sec:proof_main_theorem}

We prove Theorem \ref{maintheo} by estimating the cardinalities of the sets $\Aad(\B)$, namely by counting lattice points in the following subsets of $\C^N$. 

For every $\ua\in\clrep^r$ and every $\B>0$, let 
\begin{equation*}\Da(\B):=\{\ux\in\C^N:|x_{\rho}|_{\ii}\geq\norm(\ua^{\drho})\ \forall\rho\in\ray, \ 
\hsup(\ux)\leq\norm(\ua^\can)\B
\}.\end{equation*}
We note that $\Aad(\B)=\Da(\B)\cap\bigoplus_{\rho\in\ray}\did_\rho\ua^\drho$ for all $\ua\in\clrep^r$ and $\ud\in\idk^N$. 
In Proposition \ref{lemnAaBint} 
we approximate the cardinality of  the sets $\Aad(\B)$ by the volume of $\Da(\B)$ with respect to the Haar measure $\dadux$ on $\C^N\cong\R^{2N}$ normalized such that the volume of a fundamental domain for the lattice $\bigoplus_{\rho\in\ray}\did_{\rho}\ua^{\drho}$ is 1.
Hence, we first compute the volume of $\Da(\B)$.

\subsection{Volume computation}
Following \cite[Notation 9.1]{MR1679841}, we consider
\begin{equation*}
\partitionfunction: \torus(\C)=(\C^\times)^{\dim\X}\to\lattR, \quad \ux\mapsto (\log|x_i|_\ii)_{1\leq i\leq\dim\X}.
\end{equation*}
For every maximal cone $\sigma\in\maxc$, let $\Csigma$ be the closure of $\partitionfunction^{-1}(-\sigma)$ in $\X(\C)$. 
\begin{remark}\label{rem:Csigma}
We observe that $\ux\in\pi^{-1}(\Csigma)$ if and only if $|\ux^{\drho-\drho(\sigma)}|_\ii\leq1$ for all $\rho\in\sigma(1)$, as the proof of \cite[Proposition 11.22]{MR1679841} works with $\R$ replaced by $\C$, and in this case $\hsup(\ux)=|\ux^{\can(\sigma)}|_\ii$ (cf.~the proof of \cite[Proposition 9.8]{MR1679841}). 
\end{remark}

\begin{prop}\label{integralevaluation}
Assume that $r>1$.
For every $\ua\in\clrep^r$ and $\ud\in\idk^N$, 
\begin{equation*}
\norm(\ud)\int_{\Da(\B)}\dadux=\frac{(2\pi)^N\#\maxc\pey(X)}{(\sqrt{\disck})^N}\B(\log\B)^{r-1}+O(\B(1+\log\B)^{r-2})
\end{equation*}
where $\pey(X)$ is the constant $\alpha_{\textit{Peyre}}(\X)$ 
defined in \cite[\S7]{MR1679841}, and the implicit constant in the error term does not depend on $\ud$.
\end{prop}
\begin{proof}
Fix $\ua\in\clrep^r$. For every $\B>0$, the subset $\Da(\B)\subset \C^N$ is bounded by Lemma \ref{11.20}. 
For every $\sigma\in\maxc$, let $\Das(\B):=\Da(\B)\cap\pi^{-1}(\Csigma)$.
Let $\sigma'\neq\sigma$ in $\maxc$ and $\rho\in\sigma(1)\smallsetminus\sigma'(1)$. 
Since $|\ux^{\drho-\drho(\sigma)}|_\ii=1$ 
for all $\ux\in\Das(\B)\cap\Dasp(\B)$ (cf.~proof of \cite[Proposition 11.22]{MR1679841}), and $\deq_{\sigma,\drho}\neq0$, 
the set $\Das(\B)\cap\Dasp(\B)$ is contained in a codimension 1 subvariety of $\C^N\cong\R^{2N}$. 
Hence, $\int_{\Das(\B)\cap\Dasp(\B)}\dadux=0$, and
\begin{equation*}
\int_{\Da(\B)}\dadux=\sum_{\sigma\in\maxc}\int_{\Das(\B)}\dadux.
\end{equation*}
Fix $\sigma\in\maxc$. By \cite[Lemma V.2.2]{MR0282947}, the lattice $\did_\rho\ua^\drho$ in $\C\cong\R^2$ has determinant $\norm(\did_\rho\ua^\drho)\sqrt{\disck}/2$ for all $\rho\in\ray$. 
Therefore,
\begin{equation*}
\int_{\Das(\B)}\dadux=\frac{2^{N}}{\norm(\ud)\norm(\ua^{\can})(\sqrt{\disck})^N}\int_{\Das(\B)}\dd\ux,
\end{equation*}
where $\dd\ux$ is the usual Lebesgue measure on $\C^N\cong\R^{2N}$.

Passing to polar coordinates with $y_{\rho}:=|x_{\rho}|_{\ii}/\norm(\ua^{\drho})$ for all $\rho\in\ray$ gives
\begin{equation*}
\int_{\Das(\B)}\dadux=\frac{(2\pi)^N}{\norm(\ud)(\sqrt{\disck})^N}\int_{\DasP(\B)}\dd\uyi,
\end{equation*}
where $\dd\uyi$ is the usual Lebesgue measure on $\R^{N}$ and $\DasP(\B)$ is the set of $\uyi\in\R^N$ that satisfy
\begin{equation*}
\min_{\rho\in\ray}\yi_\rho\geq1,\ \prod_{\rho\in\ray\smallsetminus\sigma(1)}\yi^{\asrho}\leq\B,\quad \yi_\rho\leq\prod_{\rho'\in\ray\smallsetminus\sigma(1)}\yi_{\rho'}^{-\deq_{\sigma,\drho}(n_{\rho'})}\ \forall\rho\in\sigma(1)
\end{equation*}
(cf.~Remark \ref{rem:Csigma}). Moreover,
  \cite[(11.41)]{MR1679841} gives
\begin{equation*}
\int_{\DasP(\B)}\dd\uyi=\pey(X)\B(\log\B)^{r-1}+O(\B(1+\log\B)^{r-2}). \qedhere
\end{equation*}
\end{proof}

As explained in the introduction, we compare the cardinality of the sets $\Aad(\B)$ with the volume of $\Da(\B)$ after intersecting both with a suitable partition of $\C^N$ in  strongly convex rational polyhedral cones. The next two lemmas are the tools to produce such a partition.

\subsection{Ideal lattices in $\C\cong\R^2$}\label{subsec:lattices}
We show how to produce a partition of $\C\cong\R^2$ in six cones generated by bases of an ideal lattice consisting of vectors of small length with respect to the determinant of the lattice, with the property that in each cone the sum of the generators is longer than their difference.

\begin{lem}\label{idlatt_partition}
Let $\aid$ be a nonzero fractional ideal of $\K$. Then there exist $\vll_1,\dots,\vll_6\in\aid$ such that, if we set $\vll_7:=\vll_1$, then 

\begin{enumerate}
\item $[0,1)\vll_i+[0,1)\vll_{i+1}$ is a fundamental domain for the lattice $\aid\subseteq\C\cong\R^2$ for all $i\in\{1,\dots,6\}$; 
\item we can write $\vll_i=(\rll_i,\thll_i)$ in polar coordinates for $i\in\{1,\dots,6\}$ so that if we set $\thll_7:=2\pi+\thll_1$, then $0\leq\thll_{i+1}-\thll_i\leq\frac\pi2$ for all $i\in\{1,\dots,6\}$;
\item $|\vll_i|_\ii\leq16\pi^{-2}\disck\norm(\aid)$ for all $i\in\{1,\dots,6\}$. 
\end{enumerate}
\end{lem}
\begin{proof}
Let $\wll_1,\wll_2\in\aid\subseteq\R^2$ be $R$-linearly independent elements such that $|\wll_1|_\ii^{1/2}$ and $|\wll_2|_\ii^{1/2}$ are  the first and the second successive minimum, respectively, of the lattice $\aid\subseteq\R^2$ with respect to the unit ball. 
Let $\wll_3:=\wll_1+\wll_2$. 
For $i\in\{1,2,3\}$, write $\wll_i=(\rll'_i,\thll'_i)$ in polar coordinates with $\thll'_i\in[0,2\pi)$.
Without loss of generality we can assume that $\frac\pi2\leq|\thll'_1-\thll'_2|<\pi$.
Then the inequalities $|\wll_1|_\ii^{1/2}\leq|\wll_2|_\ii^{1/2}\leq|\wll_1+\wll_2|_\ii^{1/2}$ force $|\thll'_3-\thll'_1|\leq\frac\pi2$ and $|\thll'_3-\thll'_2|\leq\frac\pi3$.
Let $\ull_1:=\wll_1$, $\ull_2:=\wll_3$, $\ull_3:=\wll_2$, $\ull_4:=-\wll_1$, $\ull_5:=-\wll_3$ and $\ull_6:=-\wll_2$.
For $i\in\{1,\dots,6\}$, let $\vll_i:=\ull_i$ if $\thll'_2-\thll'_1>0$, and $\vll_i:=\ull_{7-i}$ if $\thll'_2-\thll'_1<0$.

We recall that the lattice $\aid\subseteq\R^2$ has determinant $\norm(\aid)\sqrt{\disck}/2$ (see \cite[Lemma V.2.2]{MR0282947} for example).
Then \cite[\S VIII.4.3]{MR0157947} gives 
\begin{equation*}|\wll_1|_\ii|\wll_2|_\ii\leq4\pi^{-2}\disck\norm(\aid)^2.\end{equation*} Hence, $|\wll_1|_\ii,|\wll_2|_\ii\leq4\pi^{-2}\disck\norm(\aid)$ 
as $|\wll_i|_\ii\geq\norm(\aid)$ for $i\in\{1,2\}$. Therefore, $|\vll_i|_\ii\leq2(|\wll_1|_\ii+|\wll_2|_\ii)\leq16\pi^{-2}\disck\norm(\aid)$ for all $i\in\{1,\dots,6\}$.
\end{proof}

\begin{lem}\label{latticesliceconsequence}
Let  $\lattice$ be a lattice in $\C\cong\R^2$. Let $\vll_1,\vll_2\in\lattice$ such that $[0,1)\vll_1+[0,1)\vll_{2}$ is a fundamental domain for $\lattice$ and there are $\theta_1,\theta_2\in\R$ with $0<\thll_{2}-\thll_1\leq\frac{\pi}2$ and $\vll_i=(\rll_i,\thll_i)$ for $i\in\{1,2\}$ in polar coordinates. Then
\begin{equation*}|\wll|_\ii\leq|\wll+x|_\ii<|\wll+\vll_{1}+\vll_{2}|_\ii
\end{equation*}
 for all
 $\wll\in\Z_{\geq0}\vll_1+\Z_{\geq0}\vll_2$ and all $x\in[0,1]\vll_1+[0,1]\vll_{2}$, $x\neq\vll_1+\vll_2$.
\end{lem}
\begin{proof}
Let $\wll\in\Z_{\geq0}\vll_1+\Z_{\geq0}\vll_2$ and $x\in[0,1]\vll_1+[0,1]\vll_{2}$, $x\neq\vll_1+\vll_2$. We can write $\wll=(\eta,\theta)$ and $x=(\eta',\theta')$ in polar coordinates, with $\theta_1\leq\theta,\theta'\leq\theta_2$. Then
\begin{equation*}
|\wll+x|_\ii=|\wll|_\ii+|x|_\ii-2\eta\eta'\cos(\pi-\theta+\theta'),
\end{equation*}
and $\cos(\pi-\theta+\theta')\leq0$, as  $|\theta'-\theta|\leq\pi/2$. Similarly,
\begin{equation*}
|\wll+\vll_1+\vll_2|_\ii\geq|\wll+x|_\ii+|\vll_1+\vll_2-x|_\ii,
\end{equation*}
and $|\vll_1+\vll_2-x|_\ii>0$ as $x\neq\vll_1+\vll_2$.
\end{proof}

Before proceeding with the comparison between the cardinality of the sets $\Aad(\B)$ and the volume of $\Da(\B)$, we estimate the number of lattice points close to the border of the domains $\Da(\B)$ where we count them.

\subsection{Boundary estimation}

First comes a technical result analogous to  \cite[Sublemma 11.24]{MR1679841}. 

\begin{lem}\label{counting_lemma}
Let $\e,\e_1,\dots,\e_r\in\Z_{\geq0}$, $\e>0$, $\B\geq1$, $\constlemma\in(0,1]$ such that $\constlemma\B\geq1$. Let $\ee_r:=\max\{\e_1,\dots,\e_r\}/\e$. Then 
\begin{equation*}
\sum\prod_{i=1}^ry_i^{\frac{\e_i}{\e}-1}\leq\max\{1,\ee_r2^{\ee_r}\}(\constlemma\B)^{1/\e}(1+\log\B)^r
\end{equation*}
where the sum runs through the set of positive integers $y_1,\dots,y_r$ that satisfy 
\begin{equation*}
\max\{y_1,\dots,y_r\}\leq\B \text{ and }\prod_{i=1}^ry_i^{\e_i}\leq\constlemma\B.\end{equation*} 
\end{lem}
\begin{proof} The proof goes by induction on $r$. 
Let $\Sum(\constlemma,\B,\e;\e_1,\dots,\e_r)$ be the sum in the statement. 
Assume first that $r=1$. If $\e_1=0$, then
\begin{equation}\label{eq:counting_lemma}
\Sum(\constlemma,\B,\e;0)\leq1+\log\B\leq(\constlemma\B)^{1/\e}(1+\log\B)\end{equation} as $\constlemma\B\geq1$. If $\e_1>0$, then
\begin{equation*}\begin{split}
\Sum(\constlemma,\B,\e;\e_1)\leq\int_0^{(\constlemma\B)^{1/\e_1}+1}y^{\frac{\e_1}{\e}-1}\dd y&\leq\frac{\e_1}{\e}2^{\frac{\e_1}{\e}}(\constlemma\B)^{1/\e}\leq
\\ &\leq\ee_12^{\ee_1}(\constlemma\B)^{1/\e}(1+\log\B)\end{split}
\end{equation*}
as $(\constlemma\B)^{1/\e_1}\geq1$ and $\B\geq1$.
Assume now that $r>1$. If $\e_r=0$, then $\ee_{r-1}=\ee_{r}$ and
\begin{equation*}\begin{split}
\Sum(\constlemma,\B,\e;\e_1,\dots,\e_r)&=\Sum(\constlemma,\B,\e;\e_1,\dots,\e_{r-1})\Sum(\constlemma,\B,\e;0)\leq\\
&\leq\max\{1,\ee_r2^{\ee_r}\}(\constlemma\B)^{1/\e}(1+\log\B)^{r-1}(1+\log\B)
\end{split}\end{equation*}
by the induction assumption and \eqref{eq:counting_lemma}.
If $\e_1,\dots,\e_r>0$, then
\begin{equation*}\begin{split}
\Sum(\constlemma,\B,\e;\e_1,\dots,\e_r)&=\sum_{y_r\leq(\constlemma\B)^{1/\e_r}}y_r^{\frac{\e_r}{\e}-1}\Sum(y_r^{-e_r}\constlemma,\B ,\e;\e_1,\dots,\e_{r-1})\leq\\
&\leq\max\{1,\ee_r2^{\ee_r}\}(\constlemma\B)^{1/\e}(1+\log\B)^{r-1}\Sum(\constlemma,\B,\e;0)
\end{split}\end{equation*}
as $(\constlemma\B)^{1/\e_r}\leq\B$ and $\ee_{r-1}\leq\ee_r$. The expected result follows by \eqref{eq:counting_lemma}.
\end{proof}

The next proposition, inspired by \cite[Lemma 11.25(b)]{MR1679841}, estimates the number of lattice points near the border of certain subdomains of $\Da(\B)$.

\begin{prop}\label{partdarho}
Assume that $r>1$.
Let $\ua\in\clrep^r$, $\ud\in\idk^N$, $\B\geq1$ and $\tilde\rho\in\ray$. 
For every $\rho\in\ray$, let $\vll_{\rho,1},\vll_{\rho,2}\in\did_\rho\ua^\drho$ that satisfy 
Lemma \ref{latticesliceconsequence} for the lattice $\did_\rho\ua^\drho$ in $\C\cong\R^2$. 
Let 
\begin{equation*}
\Fund:=\bigoplus_{\rho\in\ray}([0,1)\vll_{\rho,1}+[0,1)\vll_{\rho,2}),
\end{equation*} 
and let $\darho(\Fund;\B)$  be the set
\begin{equation*}
\{\ux\in\Aad(\B)\cap\bigoplus_{\rho\in\ray}(\Z_{\geq0}\vll_{\rho,1}+\Z_{\geq0}\vll_{\rho,2}):\hsup(\ux+\g)>\norm(\ua^\can)\B\},
\end{equation*}
where $\g=(\g_\rho)_{\rho\in\ray}\in\C^N$ is defined by $\g_\rho:=0$ if $\rho\neq\tilde\rho$, $\g_{\tilde\rho}:=\vll_{\tilde\rho,1}+\vll_{\tilde\rho,2}$.
Then there exist positive constants $\constB_1$ and $\constantpartdarho$, both independent of $\B$ and $\ud$, such that
\begin{gather*}
\#\darho(\Fund;\B)\leq\constantpartdarho\norm(\ud)^{-1}\B(1+\log\B)^{r-2}\min\{\norm(\ud), 1+\log\B\}.
\end{gather*}
for all $\B\geq \constB_1$.
\end{prop}
\begin{proof}
For every $\sigma\in\maxc$, let $\darhos$ be the set of $\ux\in\darho(\Fund;\B)$ such that $\pi(\ux+\g)\in\Csigma$. Then $\#\darho(\Fund;\B)\leq\sum_{\sigma\in\maxc}\#\darhos$ and it suffices to show that for every $\sigma\in\maxc$ there are positive constants $\constB'$ and  $\constant'$, both independent of $\B$ and $\ud$,  such that for all $\B\geq\constB'$,
\begin{equation*}
\#\darhos\leq\constant'\norm(\ud)^{-1}\B(1+\log\B)^{r-2}\min\{\norm(\ud),1+\log\B\}.
\end{equation*}

Fix $\sigma\in\maxc$. 
 If $\ux\in\darho(\Fund;\B)$, then $\pi(\ux+\g)\in\Csigma$ if and only if 
 \begin{equation*}
 \prod_{\rho'\in\ray}|x_{\rho'}+\g_{\rho'}|_\ii^{\deq_{\sigma,\drho}(n_{\rho'})}\leq1
 \end{equation*} for all $\rho\in\sigma(1)$, and in this case 
 \begin{equation*}\hsup(\ux+\g)=\prod_{\rho\in\ray\smallsetminus\sigma(1)}|x_\rho+\g_\rho|_\ii^\asrho,\end{equation*}
 as recalled in Remark \ref{rem:Csigma}. Therefore, if $\astrho=0$ then $\darhos=\emptyset$ for all $\B>0$. Hence, we assume that $\astrho>0$.
 Thus, $\darhos$ is the set of $\ux\in\bigoplus_{\rho\in\ray}(\Z_{\geq0}\vll_{\rho,1}+\Z_{\geq0}\vll_{\rho,2})$ that satisfy
 \begin{align}
 &|x_{\rho}|_{\ii}\geq\norm(\did_\rho\ua^{\drho})\quad \forall\rho\in\ray,\label{darhos_cond_one}\\
 &|x_\rho|_{\ii}\leq\prod_{\rho'\in\ray\smallsetminus\sigma(1)}|x_{\rho'}+\g_{\rho'}|_{\ii}^{-\deq_{\sigma,\drho}(n_{\rho'})}\quad \forall\rho\in\sigma(1),\label{darhos_cond_two}\\
  &\hsup(\ux)\leq\norm(\ua^\can)\B,\label{darhos_cond_three}\\
  &\prod_{\rho\in\ray\smallsetminus\sigma(1)}|x_\rho+\g_\rho|_\ii^\asrho>\norm(\ua^\can)\B.\label{darhos_cond_four}
 \end{align}
 Let $\rminus:=\ray\smallsetminus(\sigma(1)\cup\{\tilde\rho\})$.
 Then Lemma \ref{11.20} , \eqref{darhos_cond_one} and  \eqref{darhos_cond_three}
 give
 \begin{equation}\label{setsum_condition}
\prod_{\rho\in\rminus}|x_\rho|_\ii\leq\frac{\norm(\ua^{\sum_{\rho\in\rminus}\drho})}{\prod_{\rho\in\ray\smallsetminus\rminus}\norm(\did_\rho)}\B \quad\text{and}\quad \prod_{\rho\in\rminus}|x_\rho|_\ii^\asrho\leq\frac{\norm(\ua^{\can-\astrho\dtrho})}{\norm(\did_{\tilde\rho})^{\astrho}}\B
 \end{equation}
 for every $\ux\in\darhos$. 
 Let $\setsum$ be the set of $(x_\rho)_{\rho\in\rminus}\in\bigoplus_{\rho\in\rminus}(\Z_{\geq0}\vll_{\rho,1}+\Z_{\geq0}\vll_{\rho,2})$ that satisfy \eqref{setsum_condition} and such that $x_\rho\neq0$ for all $\rho\in\rminus$.

 For every $\ux\in\darhos$, condition \eqref{darhos_cond_three} gives 
 \begin{equation}\label{eq:condition_xrhotilde}
 |x_{\tilde\rho}|_\ii\leq\left(\frac{\norm(\ua^\can)\B}{\prod_{\rho\in\rminus}|x_\rho|_\ii^\asrho}\right)^\frac1{\astrho}.
 \end{equation}  
Moreover,  \eqref{darhos_cond_four} can be written as
 \begin{equation*}
|x_{\tilde\rho}+\g_{\tilde\rho}|_\ii^\astrho \prod_{\rho\in\rminus}|x_\rho|_\ii^\asrho>\norm(\ua^\can)\B,
 \end{equation*} and together with the triangular inequality for $|\cdot|_\ii^{1/2}$ gives
 \begin{equation}\label{eq:condition_xrhotilde_second}
 |x_{\tilde\rho}|_\ii>\left(\left(\frac{\norm(\ua^\can)\B}{\prod_{\rho\in\rminus}|x_\rho|_\ii^\asrho}\right)^{\frac1{2\astrho}}-|\g_{\tilde\rho}|_\ii^{\frac12}\right)^2.
 \end{equation}
 Therefore, 
  for every $(x_\rho)_{\rho\in\rminus}\in\setsum$ there are at most
 \begin{equation*}
\frac{\pi|\g_{\tilde\rho}|_\ii^{\frac12}}{\norm(\did_{\tilde\rho}\ua^\dtrho)}\left(\frac{\norm(\ua^\can)\B}{\prod_{\rho\in\rminus}|x_\rho|_\ii^\asrho}\right)^\frac1{2\astrho}
\end{equation*}
 elements of $\Z_{\geq0}\vll_{\tilde\rho,1}+\Z_{\geq0}\vll_{\tilde\rho,2}$ that satisfy the conditions \eqref{eq:condition_xrhotilde} and \eqref{eq:condition_xrhotilde_second}. 
  
 By \eqref{darhos_cond_two}, for every $(x_\rho)_{\rho\in\rminus}\in\setsum$ and $x_{\tilde\rho}$ as above, there are at most
 \begin{equation*}
 \boundnumbersigma:=\prod_{\rho\in\sigma(1)}\left(\frac{\constantO}{\norm(\did_\rho\ua^\drho)}\prod_{\rho'\in\ray\smallsetminus\sigma(1)}|x_{\rho'}+\g_{\rho'}|_{\ii}^{-\deq_{\sigma,\drho}(n_{\rho'})}\right)
 \end{equation*}
 elements $(x_{\rho})_{\rho\in\sigma(1)}\in\bigoplus_{\rho\in\sigma(1)}(\Z_{\geq0}\vll_{\rho,1}+\Z_{\geq0}\vll_{\rho,2})$ such that $(x_\rho)_{\rho\in\ray}\in\darhos$, where $\constantO=(\pi+64\disck\pi^{-1})/2$.
 The equality 
$ \sum_{\rho\in\sigma(1)}\drho(\sigma)=\can(\sigma)-\sum_{\rho\in\ray\smallsetminus\sigma(1)}\drho$ together with the triangular inequality for $|\cdot|_\ii^{1/2}$, Lemma \ref{idlatt_partition} and \eqref{darhos_cond_one} gives
 \begin{equation*}\begin{split}
 \boundnumbersigma&=\frac{\constantO^{N-r}}{\prod_{\rho\in\sigma(1)}\norm(\did_\rho\ua^\drho)}\prod_{\rho\in\ray\smallsetminus\sigma(1)}|x_{\rho}+\g_{\rho}|_{\ii}^{\asrho-1}=\\
&=\frac{\constantO^{N-r}}{\prod_{\rho\in\sigma(1)}\norm(\did_\rho\ua^\drho)}\left(\frac{|x_{\tilde\rho}+\g_{\tilde\rho}|_\ii}{|x_{\tilde\rho}|_\ii}\right)^{\astrho-1}\prod_{\rho\in\ray\smallsetminus\sigma(1)}|x_{\rho}|_{\ii}^{\asrho-1}\leq\\
&\leq\frac{\constantO^{N-r}}{\prod_{\rho\in\sigma(1)}\norm(\did_\rho\ua^\drho)}\left(1+\frac{4\disck^{\frac12}}{\pi}\right)^{2(\astrho-1)}\prod_{\rho\in\ray\smallsetminus\sigma(1)}|x_{\rho}|_{\ii}^{\asrho-1}.
 \end{split}
 \end{equation*} 
 Then \eqref{eq:condition_xrhotilde} above gives
 \begin{equation*}
  \boundnumbersigma\leq\frac{\constantO^{N-r}(\norm(\ua^\can)\B)^\frac{\astrho-1}{\astrho}}{\prod_{\rho\in\sigma(1)}\norm(\did_\rho\ua^{\drho})}\left(1+\frac{4\disck^{\frac12}}{\pi}\right)^{2(\astrho-1)}\prod_{\rho\in\rminus}|x_{\rho}|_{\ii}^{\frac{\asrho}{\astrho}-1}.
 \end{equation*}
 Hence, there is a positive constant $\constant$ independent of $\B$ and $\ud$ such that  for every $(x_{\rho})_{\rho\in\rminus}\in\setsum$ there are at most 
 \begin{equation*}
 \frac{\constant|\g_{\tilde\rho}|_\ii^{\frac12} \B^{1-\frac1{2\astrho}}}{\prod_{\rho\in\ray\smallsetminus\rminus}\norm(\did_\rho)}\prod_{\rho\in\rminus}|x_\rho|_\ii^{\frac{\asrho}{2\astrho}-1}
 \end{equation*}
  elements $
  (x_\rho)_{\rho\in\ray\smallsetminus\rminus}\in\bigoplus_{\rho\in\ray\smallsetminus\rminus}(\Z_{\geq0}\vll_{\rho,1}+\Z_{\geq0}\vll_{\rho,2})$
  such that $(x_\rho)_{\rho\in\ray}\in\darhos$.
 
For every $\rho\in\rminus$, let $y_\rho:=|x_\rho|_\ii/\norm(\did_\rho\ua^\drho)$.
Then
\begin{equation}\label{eq:darhos_inequality}
\#\darhos\leq\frac{(2\constantO)^{r-1}\constant\norm(\ua^{\sum_{\rho\in\rminus}(\frac{\asrho}{2\astrho}-1)\drho})|\g_{\tilde\rho}|^{\frac12}_\ii}{\norm(\ud)\prod_{\rho\in\rminus}\norm(\did_\rho)^{-\frac{\asrho}{2\astrho}}}\B^{1-\frac1{2\astrho}}\sum\prod_{\rho\in\rminus}y_\rho^{\frac{\asrho}{2\astrho}-1},
\end{equation}
where the sum runs through the set of $(y_\rho)_{\rho\in\rminus}\in(\Z_{>0})^{r-1}$ that satisfy $\prod_{\rho\in\rminus}y_\rho\leq\norm(\ud)^{-1}\B$ and $
\prod_{\rho\in\rminus}y_\rho^\asrho\leq(\prod_{\rho\in\ray\smallsetminus\sigma(1)}\norm(\did_\rho)^{-\asrho})\B$.
If $\Aad(\B)=\emptyset$, then $\darhos=\emptyset$. Hence, we assume that $\Aad(\B)\neq\emptyset$. Then $\norm(\ud)\leq\B$ by Proposition \ref{finiteness}, and $\prod_{\rho\in\ray\smallsetminus\sigma(1)}\norm(\did_\rho)^{\asrho}\leq\B$ by \eqref{darhos_cond_one} and \eqref{darhos_cond_three} and the fact that $\asrho=0$ for all $\rho\in\sigma(1)$.

If $\norm(\ud)\prod_{\rho\in\ray\smallsetminus\sigma(1)}\norm(\did_\rho)^{-\asrho}\leq1$, by Lemma \ref{counting_lemma} and Lemma \ref{idlatt_partition} there is a positive constant $\constant''$ independent of $\B$ and $\ud$ such that 
 \begin{equation*}
 \#\darhos
 \leq
 \constant''\norm(\ud)^{-1}\B(1+\log\B)^{r-1}. 
\end{equation*} 

If $\norm(\ud)\prod_{\rho\in\ray\smallsetminus\sigma(1)}\norm(\did_\rho)^{-\asrho}\geq1$, then the inequality \eqref{eq:darhos_inequality} holds with the sum running through the set of $(y_\rho)_{\rho\in\rminus}\in(\Z_{>0})^{r-1}$ that satisfy 
\begin{equation*}\prod_{\rho\in\rminus}y_\rho, \quad\prod_{\rho\in\rminus}y_\rho^\asrho\quad\leq\quad\B\prod_{\rho\in\ray\smallsetminus\sigma(1)}\norm(\did_\rho)^{-\asrho}. \end{equation*}
By Lemma \ref{counting_lemma} and Lemma \ref{idlatt_partition} there exists a positive constant $\constant'''$ independent of $\B$ and $\ud$ such that
 \begin{equation*}
 \#\darhos
 \leq
 \constant'''\norm(\ud)^{-1}\B(1+\log\B)^{r-1}. 
\end{equation*}

For every $\rho\in\rminus$, let $z_\rho:=|x_\rho|_\ii/\norm(\ua^\drho)$. We use Lemma \ref{idlatt_partition} to estimate $|\g_{\tilde\rho}|_\ii$. Then
\begin{equation*}
\#\darhos\leq
\frac{(2\constantO)^{r-1}\constant4\pi^{-1}\disck^{\frac12}\norm(\did_{\tilde\rho}\ua^{\dtrho})^{\frac12}}
{\norm(\ua^{\sum_{\rho\in\rminus}(1-\frac{\asrho}{2\astrho})\drho})\prod_{\rho\in\ray\smallsetminus\rminus}\norm(\did_\rho)}
\B^{1-\frac1{2\astrho}}\sum\prod_{\rho\in\rminus}z_\rho^{\frac{\asrho}{2\astrho}-1},
\end{equation*}
where the sum runs through the set of $(z_\rho)_{\rho\in\rminus}\in(\Z_{>0})^{r-1}$ that satisfy 
\begin{equation*}\prod_{\rho\in\rminus}z_\rho, \quad\prod_{\rho\in\rminus}z_\rho^\asrho\quad\leq\quad\B. \end{equation*}
We recall that $\norm(\did_{\tilde\rho})^{1/2}\prod_{\rho\in\ray\smallsetminus\rminus}\norm(\did_\rho)^{-1}\leq1$. By \cite[Sublemma 11.24]{MR1679841} 
there exist positive constants $\constB'\geq1$ and $\constant''''$, both independent of $\B$ and $\ud$, such that
 \begin{equation*}
 \#\darhos
 \leq
 \constant''''\B(1+\log\B)^{r-2} \quad \text{for all }\B\geq\constB'.
\end{equation*} Take $\constant':=\max\{\constant'',\constant''',\constant''''\}$.
\end{proof}

\subsection{Lattice point counting}
The next proposition compares the cardinality of the sets $\Aad(\B)$ with the volume of $\Da(\B)$ computed above.

\begin{prop}\label{lemnAaBint}
Assume that $r>1$. Let $\ua\in\clrep^r$, $\ud\in\idk^N$. Then there are positive constants $\constB_2$ and $\constantAa$, both independent of $\ud$ and of $\B$, such that 
\begin{equation*}
\left|\nAad(\B)-\int_{\Da(\B)}\dadux\right|\leq\constantAa\norm(\ud)^{-1}\B(1+\log\B)^{r-2}\min\{\norm(\ud),1+\log\B\}.
\end{equation*}
for all $\B\geq\constB_2$. 
\end{prop}
\begin{proof}
If $\norm(\ud)>\B$, then $\nAad(\B)=0$ by Proposition \ref{finiteness}. By Proposition \ref{integralevaluation} there are positive constants
$\constB'$ and $\constt'$, both independent of $\ud$ and of $\B$, such that $\int_{\Da(\B)}\dadux\leq\constt'\norm(\ud)^{-1}\B(\log\B)^{r-1}$ for all $\B\geq\constB'$. 

Assume now that $\norm(\ud)\leq\B$. 
For every $\rho\in\ray$, let
$\vll_{\rho,1},\dots,\vll_{\rho,6}\in\did_\rho\ua^{\drho}$ that satisfy Lemma \ref{idlatt_partition}, and let  $\vll_{\rho,7}:=\vll_{\rho,1}$.
Let $\NN:=\{1,\dots,6\}^{\ray}$. 

For every $\ui\in\NN$, let $\SBi:=\Da(\B)\cap\bigoplus_{\rho\in\ray}(\R_{\geq0}\vll_{\rho,i_\rho}+\R_{\geq0}\vll_{\rho,i_\rho+1})$. Then,
\begin{equation}\label{eq:partition_expression_integral}
\int_{\Da(\B)}\dadux=\sum_{\ui\in\NN}\int_{\SBi}\dadux.
\end{equation}

Recall that $\Aad(\B)=\Da(\B)\cap\bigoplus_{\rho\in\ray}\did_{\rho}\ua^{\drho}$.
Fix $\ui\in\NN$. Let $\Aadi:=\Aad(\B)\cap\SBi$. Since $\bigcup_{\ui\in\NN}\SBi=\Da(\B)$, we can compute $\nAad(\B)$ using the inclusion-exclusion principle.
For every $\ui,\ui'\in\NN$, $\ui\neq\ui'$, there exists $\rho\in\ray$ and $\underline j\in\{\ui,\ui'\}$ such that $\Aadi\cap A_{\ui'}\subseteq L_{\underline j,\rho}(\B)$, where
\begin{equation*}
L_{\underline j,\rho}(\B):=(\Z_{\geq0}\vll_{\rho,j_{\rho}+1}\oplus\bigoplus_{\rho'\in\ray\smallsetminus\{\rho\}}(\Z_{\geq0}\vll_{\rho',j_{\rho'}}+\Z_{\geq0}\vll_{\rho',j_{\rho'}+1}))\cap\Da(\B).
\end{equation*}
Then
\begin{equation}\label{eq:AadB-sum}
\left|\nAad(\B)-\sum_{\ui\in\NN}\NaBid\right|\leq2^{\#\NN}\max_{\ui\in\NN,\rho\in\ray}\#\Ri,\end{equation}
and
\begin{equation*}
\left|\nAad(\B)-\int_{\Da(\B)}\dadux\right|\leq2^{\#\NN}\max_{\ui\in\NN,\rho\in\ray}\#\Ri+\sum_{\ui\in\NN}\left|\NaBid-\int_{\SBi}\dadux\right|.
\end{equation*}

Fix $\ui\in\NN$. 
We compare $\NaBid$ with the volume of $\SBi$ by counting the number of translated fundamental domains 
\begin{equation*}\Funi:=\bigoplus_{\rho\in\ray}([0,1)\vll_{\rho,i_{\rho}}+[0,1)\vll_{\rho,i_\rho+1})
\end{equation*}
 of $\bigoplus_{\rho\in\ray}\did_\rho\ua^\drho$ contained in $\SBi$ and those that intersect the boundary 
 \begin{equation*}
 \boundarySBi:=\{\ux\in\SBi:\hsup(\ux)=\norm(\ua^\can)\B\}.
 \end{equation*}
 Let $\ux$ be an element of \begin{equation*}
 \Li:=\bigoplus_{\rho\in\ray}(\Z_{\geq0}\vll_{\rho,i_{\rho}}+\Z_{\geq0}\vll_{\rho,i_{\rho}+1}).\end{equation*} 
By Lemma \ref{latticesliceconsequence}, 
\begin{equation*}|x_\rho|_\ii\leq|x_\rho+x'_\rho|_\ii<|x_\rho+\vll_{\rho,i_\rho}+\vll_{\rho,i_\rho+1}|_\ii\end{equation*} for all $\rho\in\ray$ and   all $\ux'\in\Funi$. Let $\gi:=(\vll_{\rho,i_\rho}+\vll_{\rho,i_\rho+1})_{\rho\in\ray}$. Then
\begin{equation}
\hsup(\ux)\leq\hsup(\ux+\ux')<\hsup(\ux+\gi)\label{hhineq}
\end{equation}
for  all $\ux'\in\Funi$. Hence,   $\ux+\Funi$ intersects $\boundarySBi$ if and only if 
\begin{equation*}
\hsup(\ux)\leq\norm(\ua^\can)\B<\hsup(\ux+\gi).
\end{equation*}
Let 
\begin{gather*}
\Iid:=\{\ux\in\Aadi:\ux+\Funi\subseteq\SBi\},\\
\FiB:=\{\ux\in\Aadi:\hsup(\ux+\gi)>\norm(\ua^\can)\B\}.\end{gather*} 
 Then $\#\Aadi=\#\Iid+\#\FiB$ and $\int_{\SBi}\dadux\geq\#\Iid$.

Write $\int_{\SBi}\dadux=\int_{\SBi'}\dadux+\int_{\SBi\smallsetminus\SBi'}\dadux$ with
\begin{equation*}
\SBi':=\SBi\cap\{\ux\in\C^N:|x_\rho|_\ii\geq\norm(\did_\rho\ua^\drho),\ \forall\rho\in\ray\}.
\end{equation*}
Then
\begin{align*}
&\int_{\SBi\smallsetminus\SBi'}\dadux\leq\int_{\{\ux\in\C^N:|x_\rho|_\ii\leq\norm(\did_\rho\ua^\drho),\ \forall\rho\in\ray\}}\dadux=(2\pi\disck^{-1/2})^N,\\
&\int_{\SBi'}\dadux\leq\#\{\ux\in\Li:(\ux+\Funi)\cap\SBi'\neq\emptyset\}.
\end{align*}

Let $\ux\in\Li$ such that $(\ux+\Funi)\cap\SBi'\neq\emptyset$. Then $\hsup(\ux)\leq\norm(\ua^\can)\B$ by \eqref{hhineq}. If $\ux\notin\Aadi$, then there exists $\tilde\rho\in\ray$ such that $x_{\tilde\rho}=0$.
Let $\underline z\in\Li$ be defined by $z_\rho:=\vll_{\rho,i_{\rho}+1}$ for all $\rho\in\ray$. Then $|z_\rho|_\ii\leq16\pi^{-2}\disck\norm(\did_\rho\ua^\drho)$ for all $\rho\in\ray$ by Lemma \ref{idlatt_partition}. 
Let $\underline y\in(\ux+\Funi)\cap\SBi'$. For all $\rho\in\ray$, then $|x_\rho|_\ii\leq|y_\rho|_\ii$ by Lemma \ref{latticesliceconsequence}, and $|y_\rho|_\ii\geq\norm(\did_\rho\ua^\drho)$ since $\underline y\in\SBi'$. Hence
\begin{gather*}
|x_\rho+z_\rho|_\ii\leq2(|x_{\rho}|_\ii+|z_\rho|_\ii)\leq2(|y_\rho|_\ii+16\pi^{-2}\disck\norm(\did_\rho\ua^\drho))\leq\\
\leq2(1+16\pi^{-2}\disck)|y_\rho|_\ii.
\end{gather*}
Then $\ux+\underline z\in L_{\ui,\tilde\rho}(\constRi\B)$, where $\constRi:=(2(1+16\pi^{-2}\disck))^{\max_{\sigma\in\maxc}\sum_{\rho\in\ray}\asrho}$, and 
\begin{equation*}
\int_{\SBi'}\dadux\leq\#\Aadi+\sum_{\rho\in\ray}\#\Ris(\constRi\B).
\end{equation*}

For every $\rho\in\ray$ and every $\ux\in\Ri$ there exists an integer $m\geq0$ such that $\ux+m\vll_{\rho,i_\rho}\in\FiB$. Hence, $\#\Ri\leq\#\FiB$.
Thus,
\begin{equation*}
\int_{\SBi}\dadux\leq\#\Iid+(2\pi\disck^{-1/2})^N+\#\FiB+N\#\Fis(\constRi\B),
\end{equation*}
and
\begin{equation*}
\left|\#\Aadi-\int_{\SBi}\dadux\right|\leq(2\pi\disck^{-1/2})^N+(N+2)\max\{\#\Fis(\B),\#\Fis(\constRi\B)\},
\end{equation*}
and
\begin{multline*}
\left|\#\Aad(\B)-\int_{\Da(\B)}\dadux\right|\leq\\
\leq\#\NN(2\pi\disck^{-1/2})^N+(2^{\#\NN}+\#\NN(N+2))\max_{\ui\in\NN}\{\#\Fis(\B),\#\Fis(\constRi\B)\}.
\end{multline*}

It remains to estimate $\max_{\ui\in\NN}\{\#\Fis(\B),\#\Fis(\constRi\B)\}$.
Let $\ui\in\NN$. For every $\ux\in\FiB$, there exists $\g\in\bigoplus_{\rho\in\ray}\{0,\vll_{\rho,i_\rho}+\vll_{\rho,i_\rho+1}\}$ and $\tilde\rho\in\ray$ such that $\ux+\g\in\darho(\Funi;\B)$. Hence, 
$\#\FiB\leq2^N\sum_{\tilde\rho\in\ray}\#\darho(\Funi;\B)$. 

By Proposition \ref{partdarho} there exist  constants $\constB''\geq1$ and $\constt''\geq\pi^N$, both independent of $\ud$ and $\B$, such that 
\begin{equation*}
\#\darho(\Funi;\B)\leq\constt''\norm(\ud)^{-1}\B(1+\log\B)^{r-2}\min\{\norm(\ud), 1+\log\B\}
\end{equation*}
for all $\B\geq\constB''$, all $\tilde\rho\in\ray$ and all $\ui\in\NN$. 
Since $\constRi\geq1$,
\begin{multline*}
\max_{\ui\in\NN}\{\#\FiB,\#\Fis(\constRi\B)\}\leq\\
\leq N2^N\constt''\norm(\ud)^{-1}\constRi\B(1+\log(\constRi\B))^{r-2}\min\{\norm(\ud),1+\log(\constRi\B)\}
\end{multline*}
for all $\B\geq\constB''$. 
Take $\constB_2:=\max\{\constB',\constB''\}$ and
\begin{equation*}
\constantAa:=\max\{\constt', (2^{\#\NN}+\#\NN(N+2))N2^{N+1}\constt''\constRi(1+\log\constRi)^{r-1}\}.\qedhere
\end{equation*}
\end{proof}

\begin{proof}[Proof of Theorem \ref{maintheo}]
For $r=1$ see \cite{MR557080}. For $r>1$, 
 Proposition \ref{cardexpr} gives
 \begin{equation*}
\card(\B)=\frac{1}{\crOk^r}\sum_{\ua\in\clrep^r}\nCa(\B).
\end{equation*}
Let $\constantAa$ and $\constB_2$ be the constants in Proposition \ref{lemnAaBint}. 
For all $\B\geq\constB_2$,   Proposition \ref{CaB}, \eqref{kappa}, Proposition \ref{lemnAaBint} and Proposition \ref{mobius_function} \ref{moblowpartsum} and \ref{mobuppartsum}  give
 \begin{equation*}\begin{split}
 &\left|\nCa(\B)-\ka\int_{\Da(\B)}\daux\right|\leq\sum_{\ud\in\idk^N}|\mob(\ud)|\left|\nAad(\B)-\int_{\Da(\B)}\dadux\right|
 \leq\\
 &\leq\constantAa\B(1+\log\B)^{r-2}\left(\sum_{\substack{\ud\in\idk^N\\\norm(\ud)\leq1+\log\B}}|\mob(\ud)|+(1+\log\B)\sum_{\substack{\ud\in\idk^N\\\norm(\ud)>1+\log\B}}\frac{|\mob(\ud)|}{\norm(\ud)}\right)\\
&=O(\B(1+\log\B)^{r-2+1/\f+\varepsilon})
\end{split} \end{equation*}
for all $\varepsilon>0$. Apply Proposition \ref{integralevaluation} with $\ud=\uo$. 
\end{proof}

\section{Compatibility with Peyre's conjecture}

We conclude this paper by showing that the leading constant $\asymptconst$ in Theorem \ref{maintheo} satisfies Peyre's conjecture \cite[Conjecture 2.3.1]{MR1340296}. 

We denote by $\adeles$ the ring of adeles of $\K$. 
Let $\X(\adeles)^0$ be the inverse image of 0 under the map 
\begin{equation*}
\X(\adeles)\to\Hom_\Z(H^2_{\text{\it{\'et}}}(\X,\G_m)/H^2_{\text{\it{\'et}}}(\K,\G_m),\Q/\Z)
\end{equation*}
in \cite[Proposition 6.7]{MR1679841}, and let $\overline{\X(\K)}$ be the closure of the diagonal embedding of $\X(\K)$ in $\X(\adeles)$. Since $\X$ satisfies weak approximation by \cite[Theorem 5.1.2]{MR1845760}, the two inclusions
\begin{equation*}
\overline{\X(\K)}\subseteq\X(\adeles)^0\subseteq\X(\adeles)
\end{equation*}
are equalities. 
We recall that since $\X$ is split there is just one isomorphism class of universal torsors over $\X$. By \cite[Remarks 6.13, 7.8]{MR1679841}, Peyre's conjecture coincides then with 
\cite[Conjectures 7.12]{MR1679841}:
\begin{equation}\label{eq:peyre_conj}
\asymptconst=\pey(\X)\tau(\X,\|\cdot\|_\X),
\end{equation}
where $\pey(\X)$ is the constant defined in \cite[\S2]{MR1340296} (cf.~\cite[\S7]{MR1679841}), and by \cite[Theorem 6.19]{MR1679841}, $\tau(\X,\|\cdot\|_\X)$ is the Tamagawa number defined in \cite[Definition 6.18]{MR1679841} 
associated to the class of the universal torsor $\pi:\torsor\to\X$ and to the adelic norm $\|\cdot\|_\X$ for $\X$ that defines $\hX$ 
(cf.~\cite[(10.4)]{MR1679841}).

We recall that $\plf$ denotes the set of finite places of $\K$. For all $\nu\in\plf$, let $\nuintegers$  be the ring of integers of $k_{\nu}$.
Since the model $\wttorsor\to\wtX$ is defined over $\Ok$, \cite[Proposition 9.14(a)]{MR1679841} and  \cite[Proposition 5.20(c)]{MR1679841}\footnote{cf.~\cite[(5.21)]{MR1679841}.} give the following expression for $\tau(\X,\|\cdot\|_\X)$: 
\begin{equation*}
\firstmeasure(\ns^1(\adeles)/\ns(\K))\borel(\X_\ii(\adeles)^0\cap\pi(\torsor(\C)))\prod_{\nu\in\plf}n_\nu(\wttorsor(\nuintegers)),
\end{equation*}
where 
$\ns\cong\G_{m,k}^r$ is the N\'eron-Severi torus of $\X$, that is, the torus dual to the geometric Picard group of $\X$; 
$\ns^1(\adeles)$ is the kernel of the homomorphism
\begin{equation}\label{eq:homomorphismT^1}
\ns(\adeles)\to \Hom(\pic(\X),\R)\cong\R^r, \ \ (x_{i,\nu})_{\substack{1\leq i\leq r\\ \nu\in\pli}}\mapsto\left(\log\prod_{\nu\in\plk}|x_{i,\nu}|_\nu\right)_{1\leq i\leq r}
\end{equation}
(cf.~\cite[(5.18)]{MR1679841}); 
$\firstmeasure$ is the Haar measure on $\ns^1(\adeles)/\ns(\K)$ induced by the Haar measure $\secondmeasure$ on $\ns(\adeles)$ under the bijection established in  \cite[p.~167]{MR1679841}, where 
$\secondmeasure$ is determined by the adelic order norm in \cite[Remarks 5.9(b)]{MR1679841} and by the convergence factors $\beta_\ii=1$, $\beta_\nu=L_\nu(1,\ns)$ for $\nu\in\plf$ defined in the proof of \cite[Lemma 5.16]{MR1679841}; 
$\borel$ is the Borel measure on $\X(\C)$ defined by the adelic norm $\|\cdot\|_\X$; 
$\X_\ii(\adeles)^0$ is the compact open subset of $\X(\C)$ such that $\X(\adeles)^0=\X_\ii(\adeles)^0\times\prod_{\nu\in\plf}\wtX(\nuintegers)$, where $\X(\adeles)^0$ is
defined in \cite[Notation 6.8]{MR1679841};
$n_\nu$ is the Borel measure on $\torsor(\K_\nu)$ defined by the $\nu$-adic norm $\|\cdot\|_{\X\to\torsor}$ of \cite[Theorem 5.17]{MR1679841}. 

We are ready to show that the constant \eqref{eq:toric_asymptconst} in Theorem \ref{maintheo} satisfies \eqref{eq:peyre_conj}.
\begin{prop}\label{prop:peyre_constant}
With the notation above\\
\begin{enumerate}[label=(\roman{*}), ref=\it{(\roman{*})}]

\item $\quad\firstmeasure(\ns^1(\adeles)/\ns(\K))=(2\pi \ncl\crOk^{-1})^r$;\\\label{const1}

\item $\quad \borel(\X_\ii(\adeles)^0\cap\pi(\torsor(\C)))=(2\pi)^{N-r}\#\maxc$;\\\label{const2}

\item$\quad
\prod_{\nu\in\plf}n_\nu(\wttorsor(\nuintegers))=\ka\disck^{-N/2}$.\\\label{const3}

\end{enumerate}
\end{prop}
\begin{proof}
We first prove \ref{const1}. According to \cite[p.~167]{MR1679841},  $\torus(\K)$ is endowed with the counting measure, and $\ns(\adeles)/\ns^1(\adeles)$ is endowed with the Haar measure pullback of the usual Lebesgue measure on $\R^r$ under the isomorphism 
induced by \eqref{eq:homomorphismT^1}. By \cite[3.28--3.30]{MR1679841} and \cite[Theorem 4.14]{MR1679841}, 
\begin{equation*}
\secondmeasure=\prod_{\nu\in\plk}\beta_\nu\omega_\nu
\end{equation*} where $\omega_\nu$ are the local Haar measures on $\ns(\K_\nu)$ canonically induced by an invariant differential form of degree $r$ on $\ns$ as in \cite[\S2.2]{MR670072} 
(cf.~the lines before Theorem 4.14 in \cite{MR1679841}).
By \cite[Lemma 5.16]{MR1679841} the convergence factors $\beta_\nu$ coincide with the canonical correcting factors for $\ns$ defined in \cite[\S3.3]{MR0124326}. 
Then, 
\begin{equation*}
\firstmeasure(\ns^1(\adeles)/\ns(\K))=c(\secondmeasure; \widehat{\ns})=\disck^{r/2}\gamma(\ns;\K/\K)\end{equation*}
by \cite[(3.2.1)]{MR0124326}, where $c(\secondmeasure; \widehat{\ns})$ and $\gamma(\ns;\K/\K)$ are the numbers defined in  \cite[\S3.2]{MR0124326} and \cite[\S3.5]{MR0124326}, respectively.
We recall that $\ns\cong\G_{m,\K}^r$ is a split torus. Therefore, denoting by $\zeta_\K$ the Dedekind zeta function of $\K$,
\begin{equation*}
\gamma(\ns;\K/\K)=\gamma(\G_{m,\K};\K/\K)^r=\left(\lim_{s\to1}(s-1)\zeta_\K(s)\right)^r=(\disck^{-1/2}2\pi\ncl\crOk^{-1})^r
\end{equation*}
by  \cite[Proposition 3.5.1, Theorem 3.5.1]{MR0124326} and \cite[Corollary to Theorem VII.6.3]{MR0234930}.

To prove \ref{const2},
we recall that $\X(\adeles)^0=\X(\adeles)$, and that $\X(\adeles)=\X(\C)\times\prod_{\nu\in\plf}\wtX(\nuintegers)$ as $\X$ is proper. 
Hence, $\X_\ii(\adeles)^0=\X(\C)$. 
Moreover, $\pi(\torsor(\C))=\X(\C)$ by \cite[(2.7.2)]{MR899402}, and $\borel(\X(\C))=(2\pi)^{N-r}\#\maxc$ by \cite[Proposition 9.16]{MR1679841}.

We now prove \ref{const3}. Let $\nu\in\plf$ and denote by $\p$ the corresponding prime ideal of $\Ok$.
By \cite[Corollary 2.15, Proposition 9.14]{MR1679841},
\begin{equation*}
n_\nu(\wttorsor(\nuintegers))=\#\wttorsor(\F_{\p})\left(\frac{\mu_\nu(\nuintegers)}{\#\F_{\p}}\right)^N
\end{equation*}
where $\mu_{\nu}$ is the Haar measure on the additive locally compact group $k_{\nu}$ normalized such that $\mu_\nu(\nuintegers)\mu_\nu(\mathcal{O}_\nu^D)=1$. Here $\mathcal{O}_\nu^D$ is the inverse different of $\nuintegers$.

For $\ux\in\F_\p^N$, let $\ux\Ok:=(x'_\rho\Ok)_{\rho\in\ray}$, where $(x'_\rho)_{\rho\in\ray}\in\Ok^N$ is a representative of the class $\ux\in(\Ok/\p)^N$. 
Then, with the notation of Subsection \ref{subsec:generalized_moebius_function},
\begin{equation*}
\#\wttorsor(\F_\p)=\sum_{\ux\in\F_\p^N}\chi((\ux\Ok)_\p)=\sum_{\ux\in\F_\p^N}\partsum\mob(\ud)\chi_{\ud}(\ux\Ok). 
\end{equation*}
Since $\mob((\p^{\exponent_{\rho}})_{\rho\in\ray})=0$ for all $N$-tuples of non-negative integers
$(\exponent_{\rho})_{\rho\in\ray}$ such that $\exponent_{\rho}\geq2$ for some $\rho\in\ray$, we can replace $\partsum$ by the sum $\sum_{\ud,\p}'$ running through the set of $\ud=(\p^{\exponent_{\rho}})_{\rho\in\ray}$ with $\exponent_\rho\in\{0,1\}$ for all $\rho\in\ray$.

Let $\rho\in\ray$. If $\did_\rho=\Ok$, then $\did_\rho\mid x'_\rho\Ok$ for all $\ux\in\F_\p^N$.  If $\did_\rho=\p$, then $\did_\rho\mid x'_\rho\Ok$ if and only if $x_\rho=0$ in $\F_\p$. Hence, given $\ud=(\p^{\exponent_{\rho}})_{\rho\in\ray}$ with $\exponent_\rho\in\{0,1\}$, 
\begin{equation*}
\sum_{\ux\in\F_\p^N}\chi_{\ud}(\ux\Ok)=\#\F_\p^{\{\rho\in\ray:\exponent_\rho=0\}}=\#\F_\p^N\norm(\ud)^{-1}.
\end{equation*} 
Then,
\begin{equation*}
n_\nu(\wttorsor(\nuintegers))=\mu_\nu(\nuintegers)^N\partsum\frac{\mob(\ud)}{\norm(\ud)}.
\end{equation*}
Moreover, $\prod_{\nu\in\plf}\mu_\nu(\nuintegers)=\disck^{-1/2}$
because of the normalization of the measures $\mu_\nu$, and 
 $\ka=\prod_{\p}\partsum\frac{\mob(\ud)}{\norm(\ud)}$ by \eqref{kappa}.
\end{proof}

\subsection*{Acknowledgements}
The author was supported by grant
  DE 1646/3-1 of the Deutsche Forschungsgemeinschaft.

\bibliographystyle{personal}

\bibliography{manin_toric}

\end{document}